\documentclass{amsart}
\usepackage{amsthm}
\usepackage{amsmath}
\usepackage{amstext}
\usepackage{amsfonts}
\usepackage{amssymb}
\usepackage{latexsym}
\usepackage{amscd}
\usepackage{xypic}
\theoremstyle{plain}
\newtheorem{theorem}{Theorem}[section]
\newtheorem{lemma}[theorem]{Lemma}
\newtheorem{proposition}[theorem]{Proposition}
\newtheorem{corollary}[theorem]{Corollary}

\theoremstyle{definition}
\newtheorem{definition}[theorem]{Definition}

\newtheorem{remark}[theorem]{Remark}

\numberwithin{equation}{theorem}

\begin{document}

\title[Quotients of schemes by $\alpha_p$ or $\mu_p$ actions in characteristic $p>0$.]{Quotients of schemes by $\alpha_p$ or $\mu_p$ actions in characteristic $p>0$.}
\author{Nikolaos Tziolas}
\address{Department of Mathematics, University of Cyprus, P.O. Box 20537, Nicosia, 1678, Cyprus}
\curraddr{Department of Mathematics, Princeton University, Fine Hall, Washington Road, Princeton, NJ 08544} 
\email{tziolas@ucy.ac.cy}
\thanks{The author is supported by a Marie Curie International Outgoing  Fellowship ,  grant no. PIOF-GA-2013-624345.}




\begin{abstract}
This paper studies schemes X defined over a field of characteristic $p>0$ which admit a nontrivial $\alpha_p$ or $\mu_p$ action. In particular,  the structure of the quotient map 
$X \rightarrow Y$ is investigated. Information on local properties of the quotient Y, as singularities and local Picard groups,  structure theorems for the quotient map and  adjunction formulas for the quotient map are obtained.  
\end{abstract}

\maketitle

\section{Introduction}
Let $X$ be a scheme defined over a field $k$ of characteristic $p>0$. Suppose that $X$ admits an $\alpha_p$ or $\mu_p$-action. Then it is well known that the  quotient $\pi\colon X \rightarrow Y$ exists~\cite{Mu70}. It is a finite purely inseparable map of degree $p$. The purpose of this paper is to study the structure of the map $\pi$ and of the quotient $Y$. 

There are many reasons for studying this problem. Whenever there is a group scheme $G$ acting on a scheme $X$ defined over a field $k$, it is natural to consider the quotient of $X$ by $G$, if the quotient exists, and then by studying the quotient map  to get information about the geometry of $X$. If $G$ is finite and commutative defined over a separably closed field, then its structure is classified~\cite{Pi05}. If the characteristic of the base field $k$ is zero, then $G$ is smooth and it is the obtained by successive extensions of the group
 schemes $\mathbb{Z}/n\mathbb{Z}$. If the base field has characteristic $p>0$, then $G$ is not necessarily smooth and it is obtained by successive extensions of the group schemes $\mathbb{Z}/n\mathbb{Z}$, $\alpha_p$ and $\mu_p$. Therefore in order to study the quotient by $G$ it is essential to understand quotients by $\mathbb{Z}/n\mathbb{Z}$, $\alpha_p$ and $\mu_p$. If $n\not= p$, then quotients by $\mathbb{Z}/n\mathbb{Z}$ are similar to those in characteristic zero which have been studied extensively. The group scheme $\mu_p$ is diagonalizable and because of this $\mu_p$-actions are much easier to describe than $\alpha_p$ or $\mathbb{Z}/p\mathbb{Z}$ actions which are hard in characteristic $p>0$.
 
 One case of particular interest is when $G$ is a subgroup scheme of the automorphism scheme $\mathrm{Aut}(X)$ of a variety $X$. The study of quotients of schemes by subgroup schemes of their automorphism group is intimately related to the study of the moduli  of such schemes. In fact my main motivation for writing this paper is the study of the moduli space and moduli  stack of canonically polarized surfaces and its compactification, the moduli space of stable surfaces. In characteristic zero it is well known that the moduli space exists and moreover the corresponding stack is Deligne-Mumford. However, in positive characteristic this is not true any more. The reason for this failure is that in positive characteristic, the automorphism scheme $\mathrm{Aut}(X)$ of a canonically polarized surface $X$, even though it  is still a finite group scheme, it is  not necessarily smooth anymore. Hence  $\alpha_p$ or $\mu_p$ naturally appear as subgroup schemes of the automorphism scheme $\mathrm{Aut}(X)$ and therefore induce nontrivial $\alpha_p$ or $\mu_p$ actions on $X$.  Therefore the study of quotients of a stable surface by an $\alpha_p$ or $\mu_p$ action is essential in order to study the moduli stack of stable surfaces in positive characteristic.
 
 This paper was written in order to provide the basic technical tools in order to study the quotient of a stable surface by an $\alpha_p$ or $\mu_p$ action with the aim to investigate the moduli stack of stable surfaces in positive characteristic. Some of the results of this paper are not new. However, I intended this paper to be self contained and down to earth so that it can be used as reference for $\alpha_p$ or $\mu_p$ actions. For this reason, and sometimes due to the lack of proper reference, I included proofs of even some known results. In any case I comment on the originality or not of the results presented in this paper and give reference for those that are not new.  A general reference for the theory of group schemes and group scheme actions is~\cite{DG70}.
 
 This paper is organised as follows.
 
Section 2 introduces some terminology useful in this paper and a general structure theorem of purely inseparable maps is presented.

In section 3 it is shown that the existence of a nontrivial $\alpha_p$ or $\mu_p$ action on a scheme $X$ is equivalent to the existence of a nontrivial global vector field $D$ on $X$ such that either $D^p=0$ or $D^p=D$, respectively.

In section 4 local properties of the quotient $Y$ of a scheme $X$ by an $\alpha_p$ or $\mu_p$-action are presented. In particular information is provided about the possible singularities of $Y$. These results are in 
Propositions~\ref{prop3},~\ref{prop4}.

In section 5 the theory of $\alpha_L$-torsors is reviewed.  $\alpha_L$-torsors were, as far as I know, originally introduced by T. Ekedhal~\cite{Ek86} and under certain conditions they  are the simplest examples of quotients by
 $\alpha_p$ or $\mu_p$-actions. Theorem~\ref{sec3-th1} presents an explicit classification of $\alpha_L$-torsors and some of their important basic properties. Some of the results of this theorem were already known and some are new.

In section 6 quotients of a scheme by a nontrivial $\mu_p$ action are studied. Theorem~\ref{sec4-th1} presents a general structure theorem of such quotients. In particular it is shown that if $\pi \colon X \rightarrow Y$ is the quotient of a normal scheme $X$ by a nontrivial $\mu_p$-action, then $X$ is the normalization of an $\alpha_L$-torsor over $Y$. Moreover, if $X$ is smooth, then locally around any connected component of the fixed locus of the action, $\pi$ is an $\alpha_L$-torsor.

In section 6 quotients of a scheme by a nontrivial $\alpha_p$ action are studied. Theorem~\ref{sec5-th1} presents a general structure theorem of such quotients. 
 
In section 7 an adjunction formula for the quotient map $\pi \colon X \rightarrow Y$ of the quotient of a scheme $X$ by a nontrivial $\alpha_p$ or $\mu_p$-action is obtained. This is the content of Theorem~\ref{adjunction}. If $X$ is normal then this is well known~\cite{R-S76}. However if $X$ is not normal this is new. Quotients of non-normal schemes by $\alpha_p$ or $\mu_p$ actions are very important from the point of view of the moduli of stable varieties since non normal stable varieties naturally appear as limits of even smooth canonically polarised varieties and therefore their geometry needs to be investigated.

\section{Preliminaries.}

\subsection{Notation-Terminology.}
In this paper, unless otherwise specified,  all schemes are defined over an algebraically closed field $k$ of characteristic $p>0$. 

Let $X$ be a scheme defined over a field $k$. The tangent sheaf of $X$ is the sheaf $\mathcal{T}_X=\mathcal{H}om_X(\Omega_{X/k},\mathcal{O}_X)$.   $Der_k(X)$ denotes the space of global $k$-derivations of $X$ (or equivalently of global vector fields). It is canonically identified with $\mathrm{Hom}_X(\Omega_X,\mathcal{O}_X)=H^0(X,\mathcal{T}_X)$.

A vector field $D \in \mathrm{Der}_k(X)$ is called of additive type if $D^p=0$ and of multiplicative type if $D^p=D$.

A prime divisor $Z$ of $X$ is called an integral divisor of $D$ if and only if locally there is a derivation $D^{\prime}$ of $X$ such that $D=fD^{\prime}$, $f \in K(X)$,  $D^{\prime}(I_Z)\subset I_Z$ and $D^{\prime}(\mathcal{O}_X) \not\subset I_Z$ ~\cite{R-S76}. 

Let $\mathcal{F}$ be a coherent sheaf on $X$. By $\mathcal{F}^{[n]}$ we denote the double dual $(\mathcal{F}^{\otimes n})^{\ast\ast}$. 

A point $P \in X$ is called a normal crossing singularity if and only if
\[
\hat{\mathcal{O}}_{X,P}\cong \frac{L[[x_1,\ldots, x_n]]}{(x_1\cdots x_k)},
\]
where $L \subset \mathcal{O}_{X,P}$ is a coefficient field of $\mathcal{O}_{X,P}$.

\subsection{General structure theory of purely inseparable maps.}

Let $X$ be a normal variety defined over an algebraically closed field $k$ of characteristic $p>0$. Let $\mathcal{F} \subset \mathcal{T}_X$ be a coherent subsheaf of the tangent sheaf $\mathcal{T}_X$ of $X$.  $\mathcal{F}$ is called saturated if $\mathcal{T}_X/\mathcal{F}$ is torsion free and a subbundle of $\mathcal{T}_X$ if $\mathcal{T}_X/\mathcal{F}$ is locally free. 

Let $\mathcal{F} \subset \mathcal{T}_X$ be a coherent subsheaf of $\mathcal{T}_X$. $\mathcal{F}$ is called a foliation on $X$ if and only if $\mathcal{F}$ is closed under Lie bracket and it is $p$-closed~\cite{Ek87}. These conditions mean that if $U \subset X$ is any open subset of $X$ and any section $D \in \mathcal{F}(U)$, then $D^p \in \mathcal{F}(U)$ and $[D,D^{\prime}]\in \mathcal{F}(U)$ for any $D^{\prime}\in \mathcal{F}(U)$, $D, D^{\prime}$ viewed as a $k$-derivations of $\mathcal{O}_U$.            
                                                                                                                                                             
Let $\mathcal{F} \subset T_X$ be a foliation on $X$. We define $Ann(\mathcal{F})\subset \mathcal{O}_X$ to be the sheaf of subrings of $\mathcal{O}_X$ on $X$ defined by 
\[
Ann(\mathcal{F}) (U)=\{a\in \mathcal{O}_X(U)|\; D(a)=0, \forall D\in \mathcal{F}(U)\},
\]
where $U \subset X$ is any open subset, and the sections of $\mathcal{F}$ over $U$ are viewed as $k$-derivations of $\mathcal{O}_U$. Then $\mathcal{O}_X^p \subset Ann(\mathcal{F})$ and hence there exists a factorization 
\begin{equation}\label{sec0-diag-1}
\xymatrix{ 
X \ar[rr]^F\ar[dr]_{\pi}  & & X^{(1)} \\
  &   Y \ar[ur]_{\phi} & \\
}
\end{equation}
where $F \colon X \rightarrow X^{(1)}$ is the relative Frobenius and $Y=\mathrm{Spec} \left( Ann(\mathcal{F}) \right)$. Conversely, let $Y$ be any scheme that fits in a commutative diagram as above and let $\mathcal{F}=Ann(\mathcal{O}_Y)\subset \mathcal{T}_X$ be defined by $\mathcal{F}(U)=\{D \in \mathcal{T}_X(U) |\; D(a)=0, \forall a\in \mathcal{O}_Y(U)\}$, for any open $U \subset X$. Then $\mathcal{F}$ is a foliation. The next proposition says that these two operations establish a one-to-one correspondence between foliations and normal varieties between $X$ and $X^{(1)}$.
  
\begin{proposition}\label{general-structure-theory}\cite{Ek87}
Let $X$ be a normal variety defined over an algebraically closed field of characteristic $p>0$. 
\begin{enumerate}
\item Let $\mathcal{F} \subset \mathcal{T}_X$ be a foliation on $X$ of rank $r$ and let $Y =\mathrm{Spec}\left(Ann(\mathcal{F})\right)$. Then $[K(X):K(Y)]=p^r$.
\item Let $\mathcal{F}\subset \mathcal{T}_X$ be a foliation on $X$ and $Y$ a normal variety between $X$ and $X^{(1)}$ as in diagram~\ref{sec0-diag-1}. Then $Ann(Ann(\mathcal{F}))=\mathcal{F}$ and $Ann(Ann(\mathcal{O}_Y))=\mathcal{O}_Y$. Hence there exists a one to one correspondence between foliations on $X$ and normal varieties $Y$ between $X$ and $X^{(1)}$. The correspondence is given by $\mathcal{F}\mapsto \mathrm{Spec}(Ann\left(\mathcal{F})\right)$ and $Y \mapsto Ann(\mathcal{O}_Y)$.
\item Suppose that $X$ is smooth. Then $Y$ is smooth if and only if $\mathcal{F}$ is a subbundle.
\item Suppose that $X$ is smooth and $\mathcal{F}$ be a subbundle of $\mathcal{T}_X$. Let $\pi \colon X \rightarrow Y =\mathrm{Spec}\left(Ann(\mathcal{F}\right)$ the natural map. Then there exists a natural exact sequence
\[
0 \rightarrow F^{\ast}(\mathcal{F}^{\ast}) \rightarrow \pi^{\ast}\Omega_{Y/k} \rightarrow \Omega_{X/k} \rightarrow \mathcal{F}^{\ast} \rightarrow 0,
\] 
where $F \colon X \rightarrow X$ is the absolute Frobenious map. Therefore 
\[
\omega_X=\pi^{\ast}\omega_Y \otimes \left(\wedge^r\mathcal{F}\right)^{p-1},
\]
where $r$ is the rank of $\mathcal{F}$.
\end{enumerate}
\end{proposition}    
The correspondence given in the previous proposition between foliations on $X$ and normal varieties between $X$ and $X^{(1)}$ is a natural extension of the Jacobson correspondence between subfields of a field $K$ containing $K^p$ and sub-Lie algebras of $Der(K)$ closed under $p$-powers~\cite{Jac64},~\cite{B03}.                                                                                                                                                                                                                                                                                                                                                                                                                                                                                                                                                                                                                                                                                                                                       

\section{Relation between $\alpha_p$, $\mu_p$ actions and derivations.}
The purpose of this section is to show the relation between nontrivial $\mu_p$ or $\alpha_p$ actions and global vector fields. 

\begin{proposition}\label{sec1-prop1}
Let $X$ be a scheme of finite type over an algebraically closed field $k$ of characteristic $p>0$. $X$ admits a nontrivial $\alpha_p$ or $\mu_p$ action if and only if $X$ has a nontrivial global vector field $D$ such that $D^p=D$ or $D^p=0$, respectively.
\end{proposition}
\begin{proof}

Suppose that $X$ has a nontrivial vector field $D$ such that either $D^p=0$ or $D^p=D$. Define the map
\[
\Phi \colon \mathcal{O}_X \rightarrow \frac{\mathcal{O}_X[t]}{(t^p)}
\]
by setting 
\[
\Phi(a)=\sum_{k=0}^{p-1}\frac{D^ka}{k!}t^k
\]
If $D$ is of additive type (i.e., $D^p=0$) then it is straightforward to check that $\Phi$ defines an action $\alpha_p$ on $X$ and if it is of multiplicative type (i.e., $D^p=D$) then it defines an action of $\mu_p$ on $X$.

Suppose now that $X$ admits a nontrivial $\alpha_p$ or $\mu_p$ action. I will show that $X$ admits a nontrivial vector field $D$ of either additive or multiplicative type, respectively. I will only do the case when an $\alpha_p$ action exists. The other is similar and is omitted. Suppose that $X$ admits a nontrivial $\alpha_p$-action. Let \[
\mu \colon \alpha_p \times X \rightarrow X
\]
be the map that defines the action. Let 
\[
\mu^{\ast} \colon \mathcal{O}_X \rightarrow \mathcal{O}_X \otimes_k \frac{k[t]}{(t^p)}
\]
be the corresponding map on the sheaf of rings level. The additive group $\alpha_p$ is $\mathrm{Spec}\frac{k[t]}{(t^p)}$ as a scheme with group scheme structure given by
\[
m^{\ast}\colon \frac{k[t]}{(t^p)} \rightarrow \frac{k[t]}{(t^p)} \otimes \frac{k[t]}{(t^p)} \]
defined by $m^{\ast}(t)=1\otimes t + t \otimes 1$. Then by the definition of group scheme action, there is a commutative diagram.
\[
\xymatrix{
\mathcal{O}_X \ar[r]^{\mu^{\ast}} \ar[d]^{\mu^{\ast}} & \mathcal{O}_X \otimes_k \frac{k[t]}{(t^p)} \ar[d]^{\mu^{\ast}\otimes 1} \\
\mathcal{O}_X \otimes_k \frac{k[t]}{(t^p)} \ar[r]^{1\otimes m^{\ast}} & \mathcal{O}_X \otimes_k \frac{k[t]}{(t^p)} \otimes_k \frac{k[t]}{(t^p)}\\
}
\]
The map $\mu^{\ast}$ is given by \[
\mu^{\ast}(a)= a\otimes 1 +\Phi_1(a)\otimes t +\sum_{k=2}^{p-1}\Phi_k(a)\otimes t^k,
\] 
where $\Phi_k \colon \mathcal{O}_X \rightarrow \mathcal{O}_X$ are additive maps. The fact that $\mu^{\ast}$ is a sheaf of rings map shows that $\Phi_1$ is a derivation which we call $D$. I will show by induction that $\Phi_k(a)=D^ka/k!$. From this it follows that any $\alpha_p$ action is induced by a global vector field $D$ as in the first part of the proof.

From the commutativity of the previous diagram and the definition of $\mu^{\ast}$ and $m^{\ast}$ it follows that
\begin{gather*}
a \otimes 1 \otimes 1 + Da \otimes t \otimes 1 +\sum_{k=2}^{p-1} \Phi_k(a) \otimes t^k \otimes 1+ Da \otimes 1 \otimes t +D^2a\otimes t \otimes t + \\
\sum_{k=2}^{p-1}\Phi_k(Da)\otimes t^k 
\otimes t + \sum_{k=2}^{p-1}\left( \Phi_k(a)\otimes 1 \otimes t^k +D(\Phi_k(a))\otimes t \otimes t^k +\sum_{s=2}^{p-1}\Phi_s(\Phi_k(a))\otimes t^s \otimes t^k\right) =\\
a\otimes 1 \otimes 1 +Da \otimes t \otimes 1 + Da \otimes 1 \otimes t +\sum_{k=2}^{p-1}\sum_{s=0}^{k}\Phi_k(a) \binom{k}{s}t^s \otimes t^{k-s}.
\end{gather*}
Equating the coefficients of $t^{k-1}\otimes t $ on both sides of the equation we get that $k\Phi_k(a)t^{k-1}\otimes t = \Phi_{k-1}(Da)\otimes t^{k-1} \otimes t$. By induction, 
$\Phi_{k-1}=D^{k-1}/(k-1)!$. Therefore $\Phi_k(a)=D^ka/k!$ as claimed. Moreover, equating the coefficients of $t \otimes t^{p-1}$ it follows that $D^p=0$.

\end{proof}


\section{Existence and basic properties of the quotient.}

Let $X$ be a scheme of finite type over a field $k$ of characteristic $p>0$. Suppose that $X$ admits a nontrivial $\alpha_p$ or $\mu_p$ action. Then it is well known~\cite{Mu70} that the quotient $\pi \colon X \rightarrow Y$ of $X$ by the $\alpha_p$ or $\mu_p$ action exists as an algebraic scheme and that $\pi$ is a finite morphism. In particular, by Proposition~\ref{sec1-prop1} the action is induced by a global derivation $D$ on $X$ such that either $D^p=0$ or $D^p=D$. Then locally if $X =\mathrm{Spec}A$ then $Y=\mathrm{Spec}B$, where $B=A^D=\{a\in A , \; Da=0\}$. Moreover, by Proposition~\ref{general-structure-theory}, $\pi$ corresponds to the foliation on $X$ defined by the saturation $\mathcal{F}$ of the subsheaf of $\mathcal{T}_X$ generated by $D$.

\subsection{Local properties of the quotient.}
The purpose of this section is to study the relation between the formation of quotients by $\alpha_p$ or $\mu_p$ actions and the operations of completion and localization.

The next proposition shows that the formation of quotient commutes with completion. This allows the calculation of the singularities of the quotient $Y$ by passing to the completion. 

\begin{proposition}\label{sec2-prop1}
Let $(A,m_A)$ be a local ring and $D\in \mathrm{Der}(A)$ a derivation of $A$ such that $D(m_A)\subset m_A$. Then $D$ extends to a derivation $\hat{D}$ of the completion $\hat{A}$ of $A$ along $m_A$ and moreover 
\[
\left( \hat{A} \right)^{\hat{D}}=\widehat{A^D},
\]
where $\widehat{A^D}$ is the completion of $A^D$ along the maximal ideal $m_A\cap A^D$.
\end{proposition}
\begin{proof}
Let $B=A^D$. Then $A$ is a finitely generated $B$-module~\cite{Mu70}. Hence $\hat{A}=A\otimes_B \hat{B}$. Let $\Phi \colon A \rightarrow A$ be defined by $\Phi(a)=Da$, $a\in A$. Then $\Phi$ is a morphism of $B$-modules. Moreover there is an exact sequence of $B$-modules
\[
0 \rightarrow B \rightarrow A \stackrel{\Phi}{\rightarrow} A
\]
Tensoring with $\hat{A}$ over $B$ and since $\hat{A}$ is a flat $B$-module it follows that the following sequence is exact
\[
0 \rightarrow \hat{B} \rightarrow \hat{A} \stackrel{\hat{\Phi}}{\rightarrow} \hat{A}
\]
But $\hat{\Phi}$ is given by $\hat{\Phi}(\hat{a})=\hat{D}(\hat{a})$, for any $\hat{a} \in \hat{A}$. Therefore $\hat{B}=\left(\hat{A}\right)^{\hat{D}}$, as claimed.

\end{proof}

The next proposition shows that the formation of quotient commutes with localization.

\begin{proposition}~\label{sec2-prop2}
Let $A$ be a ring and $D\in \mathrm{Der}(A)$ a derivation of $A$. Let $P \subset A$ be a prime ideal. Then
\[
\left( A_P\right)^D=\left( A^D\right)_Q,
\]
where $Q=A^D \cap P$.
\end{proposition}

\begin{proof}
Clearly $\left( A^D\right)_Q \subset \left( A_P\right)^D$. Let $a/s \in \left( A_P\right)^D$. Then $D(a/s)=0$. But $a/s=(as^{p-1})/s^p$. Now $s^p \in A^D$, $s^p \not\in P \cap A^D=Q$ and 
\[
D(as^{p-1})=D(as^p/s)=s^p D(a/s)=0.
\]
Hence $a/s \in \left( A^D\right)_Q $ and therefore
\[
\left( A_P\right)^D=\left( A^D\right)_Q,
\]
as claimed.
\end{proof}

\begin{definition}\label{sec2-def1}
Let $A$ be a ring and $D\in \mathrm{Der}(A)$ a derivation of $A$. Let $B=A^D=\{a\in A |\; Da=0\}$. Then for any $k \in \{0,1,\ldots,p-1\}$ we define
\begin{enumerate}
\item \[
L_k(D)=\{a\in A |\; Da=ka \},
\]
\item \[
E_k(D)=\{a\in A |\; D^ka=0\}.
\]
\end{enumerate}
The sets $L_k(D)$ and $E_k(D)$ have natural structures of $B$-modules. 
\end{definition}\label{def-of-L-N}
As will be shown in sections~\ref{mu-p} and~\ref{alpha-p}, the $B$-modules $L_k(D)$ and $E_k(D)$ carry very important information about the structure of the quotient $B$. The first one in the case when $D^p=D$ and hence induces an action of $\mu_p$ on $\mathrm{Spec} A$ and $E_k(D)$ in the case when $D^p=0$ and hence induces an action of $\alpha_p$ on $\mathrm{Spec} A$. 

The next proposition shows that $L_k(D)$ and $E_k(D)$ commute with completion and localization. This allows their calculation in local coordinates.

\begin{proposition}\label{subsec2-prop3}
Let $A$ be a Noetherian ring  and $D\in \mathrm{Der}(A)$ a derivation of $A$. Let $P \subset A$ be a prime ideal of $A$. Then
\begin{enumerate}
\item $D$ lifts naturally to a derivation $D_P$ of the localization $A_P$ and moreover
\begin{gather*}
L_k(D_P) \cong (L_k(D))_Q,\\
E_k(D_P) \cong (E_k(D))_Q,
\end{gather*}
where $Q=A^D\cap P$.
\item If $D(P) \subset P$, then $D$ lifts to a derivation $\hat{D}$ of the completion $\hat{A}$ of $A$ at $P$ and moreover,
\begin{gather*}
\widehat{L_k(D)}\cong L_k(\hat{D}),\\
\widehat{E_k(D)}\cong E_k(\hat{D})
\end{gather*}
\end{enumerate} 
\end{proposition}
\begin{proof}
Consider the exact sequence
\[
0 \rightarrow M \rightarrow A \stackrel{\phi}{\rightarrow} A
\]
given by either $\phi(a)=Da-ka$, or $\phi(a)=D^ka$. In both cases $\phi$ is an $A^D$-module homomorphism. Since $A$ is Noetherian and finitely generated as a $B$-module, it follows that $M$ is a finitely generated $B$-module as well. Hence $\hat{M}=M\otimes_B \hat{B}$ and $\hat{A}=A\otimes_B \hat{B}$. 

If $\phi(a)=Da-ka$ then $M=L_k(D)$ and if $\phi(a)=D^ka$ then $M=E_k(D)$. The proposition is proved by taking completion (or localization) on the above exact sequence and considering that completion and localization are exact functors.

\end{proof}

\subsection{Singularities of the quotient.}
 The purpose of this section is to describe the singularities of  the quotient $Y$ of a scheme $X$ of finite type over a field $k$ by a nontrivial $\alpha_p$ or $\mu_p$ action induced by a nontrivial vector field $D$ such that either $D^p=0$ or $D^p=D$.

\begin{definition}\cite{Sch07}
\begin{enumerate}
\item The fixed locus of the action of $\alpha_p$ or $\mu_p$ (or of $D$) on $X$ is the closed subscheme of $X$ defined by the ideal sheaf generated by $D(\mathcal{O}_X)$.
\item A point $P \in X$ is called an isolated singularity of $D$ if there is an embeded component $Z$ of the fixed locus of $D$ such that $P \in Z$. The vector field $D$ is said to have only divisorial singularities if the ideal $D(\mathcal{O}_X)$ has no embedded components.
\end{enumerate}
\end{definition}

The next proposition gives some first information about the singularities of $Y$. 
\begin{proposition}\label{prop3}
Let $X$ be an integral scheme of finite type over an algebraically closed field of characteristic $p>0$. Suppose $X$ has an $\alpha_p$ or $\mu_p$ action induced by a vector field $D$ of either additive or multiplicative type. Let $\pi \colon X \rightarrow Y$ be the quotient. Then
\begin{enumerate}
\item If $X$ is normal then $Y$ is normal.
\item If $X$ is $S_2$ then $Y$ is $S_2$ as well.~\cite{Sch07}
\item If $X$ is smooth then the singularities of $Y$ are exactly the image of the embedded part of the fixed locus of the ~\cite{AA86}
\item If $X$ is normal and $\mathbb{Q}$-Gorenstein, then $Y$ is also $\mathbb{Q}$-Gorenstein. In particular, let $D$ be a divisor in $Y$ and $\tilde{D}$ be the divisorial part of $\pi^{-1}(D)$. Then if $n\tilde{D}$ is Cartier, $pnD$ is Cartier too.
\end{enumerate}
\end{proposition}

\begin{proof}
Normality of $Y$ is a local property so we may assume that $X$ and $Y$ are affine. Let $X=\mathrm{Spec}A$ and $Y=\mathrm{Spec}B$, where $B=\{a\in A, \; Da=0\} \subset A$. Let $\bar{B} \subset K(B)$ be the integral closure of $B$ in its function field $K(B)$. Let $z=b_1/b_2 \in \bar{B}$. Then since $K(B)\subset K(A)$ and $A$ is normal, $z\in A$. But $Dz=D(b_1/b_2)=(b_2Db_1-b_1Db_2)/b_2^2=0$. Therefore $z\in B$ and hence $B$ is integrally closed.

Suppose that $X$ is $\mathbb{Q}$-Gorenstein. Let $D$ be a divisor on $Y$. The property that $D$ is $\mathbb{Q}$-Cartier is local so we may assume that $X$ and $Y$ are affine, say $X=\mathrm{Spec}A$ and $Y=\mathrm{Spec}B$. Then $D$ is $\mathbb{Q}$-Cartier if and only if $nD=0$ in $\mathrm{Cl}(B)$, for some $n \in \mathbb{N}$. Consider the natural map $\phi\colon \mathrm{Cl} (B) \rightarrow \mathrm{Cl} (A)$. Then according to~\cite{Fo73}, 
\[
\mathrm{Ker}\phi \subset  H^1(G, A^{\ast}+K(A)t)
\]
where $G$ is the additive subgroup of $Der_k(A)$ generated by $D$, $A^{\ast}+K(A)t \subset K(A)[t]/(t^2)$ is the multiplicative subgroup and $G$ acts on it by the usual automorphisms induced by $D$. Then since $k$ has characteristic $p>0$, $G \cong \mathbb{Z}/p\mathbb{Z}$  and therefore $H^1(G, A^{\ast}+K(A)t)$ is $p$-torsion. Therefore $\mathrm{Ker}\phi$ is $p$-torsion as well. Hence if $n\tilde{D}$ is $\mathbb{Q}$-Cartier, then $nD\in \mathrm{Ker}\phi$ and hence $pnD=0$ in $\mathrm{Cl} (B)$. Therefore $pnD$ is Cartier as claimed. 
\end{proof}

Even if $X$ is smooth, it is very hard to give more detailed information about the singularities of $Y$, even more to classify them. The difficulty arises mainly from the complex structure of $\alpha_p$ actions and quotients. In this case the quotient may not even have rational singularities (such examples can be found in~\cite{Li08}). Quotients by $\mu_p$ are much easier to describe. In this case the quotient $Y$ has cyclic quotient singularities of type 
$\frac{1}{p}(1,m)$~\cite{Hi99}. The characteristic 2 case is also simpler (as is probably expected from the characteristic zero case of quotients by $\mathbb{Z}/2\mathbb{Z}$). 

The next proposition provides some more information about the singularities of the quotient. 

\begin{proposition}\label{prop4}
Let $X$ be a smooth surface with a nontrivial $\alpha_p$ or $\mu_p$ action induced by a global vector field $D$ or either additive or multiplicative type. Let $\pi \colon X \rightarrow Y$ be the quotient. Then there exists a commutative diagram 
\begin{gather}\label{prop4-diagram}
\xymatrix{
X^{\prime}\ar[r]^f \ar[d]^{\pi^{\prime}} & X \ar[d]^{\pi} \\
Y^{\prime} \ar[r]^g & Y \\
}
\end{gather}
such that 
\begin{enumerate}
\item $g \colon Y^{\prime} \rightarrow Y$ is the minimal resolution of $Y$, $X^{\prime}$ is normal and $f$ is birational.
\item The vector field $D$ on $X$ lifts to a global vector field $D^{\prime}$ on $X^{\prime}$ and $\pi^{\prime} \colon X^{\prime} \rightarrow Y^{\prime}$ is the quotient of $X^{\prime}$ by $D^{\prime}$.
\end{enumerate}
Suppose that  $p=2$. Then in addition to the above,
\begin{enumerate}\setcounter{enumi}{2}
\item $Y$ is Gorenstein. Moreover, if $Y$ has canonical singularities, then $Y$ has singularities of type either $A_1$, $D_{2n}$, $E_7$ or $E_8$.
\item If $D$ is of multiplicative type then $X^{\prime}$ is smooth and $f$ is obtained by blowing up the isolated singular points of $D$~\cite{Hi99}.
\item If $D$ is of additive type, then a diagram like (\ref{prop4-diagram}) exists where both $X^{\prime}$ and $Y^{\prime}$ are smooth. However, $Y^{\prime}$ is not necessarily the minimal resolution of $Y$\cite{Hi99}.
\end{enumerate}
\end{proposition}

\begin{proof}
The only parts of the proposition that need to be proved are the parts 1-3. Let $g \colon Y^{\prime} \rightarrow Y$ be the minimal resolution of $Y$. Let $\pi^{\prime} \colon X^{\prime}\rightarrow Y^{\prime}$ be the normalization of $Y^{\prime}$ in the field of fractions $K(X)$ of $X$. Then $\pi^{\prime}$ is a purely inseparable map of degree $p$. I will show that there exists a map $f \colon X^{\prime} \rightarrow X$ such that $\pi f =g\pi^{\prime}$ giving rise to the diagram~\ref{prop4-diagram}. The rational map $X\dasharrow Y^{\prime}$ defined by $\pi$ and $g$ is resolved after a sequence of blow ups of $X$. Therefore there exists a commutative diagram
 \[
\xymatrix{
 X^{\prime}\ar[dr]_{\pi^{\prime}} &   Z\ar[l]_{\psi}\ar[r]^{\phi} \ar[d]^{\delta} & X \ar[d]^{\pi} \\
       &Y^{\prime} \ar[r]^g & Y \\
}
\]
 where $\phi$ is a sequence of blow ups resolving $X\dasharrow Y^{\prime}$ and $\psi \colon Z \rightarrow X^{\prime}$ is the factorization of $\delta$ through $X^{\prime}$ which exists  since $Z$ is normal in $K(X)$. Since 
 $\pi^{\prime}$ and $\pi$ are finite morphisms, it follows that every $\psi$-exceptional curve is also a $\phi$-exceptional curve. Therefore the rational map $f=\phi\psi^{-1}$ is in fact a morphism and hence there exists a commutative diagram as claimed in~\ref{prop4}.1.

Next I will show that the vector field $D$ of $X$ lifts to a vector field $D^{\prime}$ of $X^{\prime}$. Since $X$ and $X^{\prime}$ are birational, $D$ gives a rational vector field $D^{\prime}$ of $X^{\prime}$. It is not hard to see that $Y^{\prime}=X^{\prime}/D^{\prime}$. Then in order to show that $D^{\prime}$ is regular it suffices to show that it has no poles. Let $\Delta$ be the divisorial part of $D$. Then~\cite{R-S76}
 \[
K_X=\pi^{\ast}K_Y+(p-1)\Delta.
\]
Moreover, since $Y^{\prime}$ is the minimal resolution of $Y$,
\[
K_{Y^{\prime}}=g^{\ast}K_Y-F,
\]
where $F$ is an effective $g$-exceptional $\mathbb{Q}$-divisor. Therefore from the commutative diagram~\ref{prop4-diagram} it follows that
\begin{gather*}
K_{X^{\prime}}=f^{\ast}K_X+E=f^{\ast}(\pi^{\ast}K_Y+(p-1)\Delta)+E=(\pi^{\prime})^{\ast}g^{\ast}K_Y+(p-1)f^{\ast}\Delta+E= \\
(\pi^{\prime})^{\ast}K_{Y^{\prime}}+(\pi^{\prime})^{\ast}F+(p-1)f^{\ast}\Delta+E,
\end{gather*}
where $E$ is an $f$-exceptional divisor. But the last adjunction formula shows that the divisor of $D^{\prime}$ is $(\pi^{\prime})^{\ast}F+(p-1)f^{\ast}\Delta+E$, which is effective. 
Hence $D^{\prime}$ has no poles and therefore it is regular. This concludes the proof of~\ref{prop4}.1 ~\ref{prop4}.2.

Suppose that $p=2$. Then $\pi$ factors through the geometric Frobenious $F \colon X \rightarrow X^{(2)}$. In fact there is a commutative diagram
\[
\xymatrix{
     & Y \ar[dr]^{\nu} \\
X \ar[ur]^{\pi}\ar[rr]^F & & X^{(2)}
}
\]
Since $X^{(2)}$ is smooth and $Y$ is normal, then $\nu$ is an $\alpha_L$-torsor over $X^{(2)}$, for some line bundle $L$ on $X^{(2)}$~\cite{Ek87}. Then since $X$ is smooth, $Y$ has hypersurface singularities and therefore it is Gorenstein. Suppose that $Y$ has canonical singularities. Then the dynking diagram of any singular point of $Y$ is of type either $A_n$, $D_n$, $E_6$, $E_7$ or $E_8$. By Proposition~\ref{prop3}.4, the local Picard groups of the singular points of $Y$ are 2-torsion. Therefore these can be only $A_1$, $D_{2n+1}$, $E_7$ or $E_8$. This shows~\ref{prop4}.3.

\end{proof}

\begin{remark}
\item Proposition~\ref{prop4}.5 essentially says that if $p=2$ then the isolated singularities of the vector field $D$ can be resolved by a sequence of blow ups. If $p>2$ then this is not possible in general. Take for example  $X=\mathbb{A}^2_k$, $p=5$ and $D=x \partial/\partial x+2y\partial/\partial y$. This is a vector field of multiplicative type. Suppose that a diagram like in Proposition~\ref{prop4} exists with both $X^{\prime}$ and $Y^{\prime}$ smooth. Then $f$ is obtained by successively blowing up the isolated singularities of $D$. Let $X_1\rightarrow X$ be the first blow up, i.e., the blow up of the singular point of $D$. Then a straightforward calculation shows that the lifting $D_1$ of $D$ on $X_1$ has exactly two isolated singular points, say $P$ and $Q$. Moreover, locally at $P$, $D_1=x   \partial/\partial x+y\partial/\partial y$ and locally at $Q$, $D_1=2(x   \partial/\partial x+2y\partial/\partial y)$. 
Hence at $Q$, $D_1$ has exactly the same form as $D$. Hence every time a singular point is blown up at which the vector field has the form $\lambda (x   \partial/\partial x+2y\partial/\partial y)$, $\lambda \in \mathbb{Z}_p$, another singular point will appear in the blow up where the lifted vector field will have the same form. Hence the process of blowing up the isolated singular points of the vector field does not lead to a vector field without isolated singular points and hence a diagram like in Proposition~\ref{prop4}  does not exist in this case.

However, even though there is no resolution in general  of the isolated singularities of $D$ by usual blow ups as in the case $p=2$, Proposition~\ref{prop4}.1,2 says that there exists a partial resolution by weighted blow ups instead, hence the singularities on $X^{\prime}$.
\end{remark}

\subsection{Relation between the divisors of $X$ and $Y$.}
Let $X$ be a scheme with a nontrivial $\alpha_p$ or $\mu_p$-action which by Proposition~\ref{sec1-prop1} is induced by a nontrivial vector field $D$ of $X$ such that either $D^p=0$ or $D^p=D$. Let $\pi \colon X \rightarrow Y$ be the quotient. 

The next propositions exhibit certain relations that exist between the divisors of $X$ and its quotient $Y$.

\begin{proposition}~\cite{R-S76}
Suppose that $X$ is a normal variety. Let $C$ be an irreducible divisor on $X$ and $C^{\prime}$ its scheme theoretic image in $Y$. If $C$ is an integral subvariety of $D$ then $\pi_{\ast}C=pC^{\prime}$ and $\pi^{\ast}C^{\prime}=C$. If on the other hand $C$ is not an integral subvariety of $D$ then $\pi_{\ast}C=C^{\prime}$ and $\pi^{\ast}C^{\prime}=pC$.
\end{proposition}

\begin{proposition}
Let $L \in \mathrm{Pic}(X)$ be a line bundle on $X$. Then there exists $M \in \mathrm{Pic}(Y)$ such that $L^p=\pi^{\ast}M$.
\end{proposition}

\begin{proof}
$L$ corresponds naturally to an element of $H^1(X,\mathcal{O}_X^{\ast})$. Let $V_i$ be an affine cover of $Y$ and $U_i=\pi^{-1}(V_i)$ the corresponding affine cover of $X$. Then $L$ is determined by elements $a_{ij} \in \mathcal{O}_{U_{ij}}^{\ast}$, $U_{ij}=U_i \cap U_j$, which satisfy the cocycle condition $a_{ij}a_{ik}^{-1}a_{jk}=1$. But $a_{ij}^p \in \mathcal{O}_{V_{ij}}$ and satisfy the cocycle condition. Hence $a_{ij}^p$ determine a line bundle $M$ on $Y$ such that $\pi^{\ast}M=L^p$. 
\end{proof}

In exactly the same way it can be proved that.

\begin{proposition}
Let $L_i \in \mathrm{Pic}(Y)$, $i=1,2$, such that $\pi^{\ast}L_1 \cong \pi^{\ast}L_2$. Then $L_1^p \cong L_2^p$.
\end{proposition}

In the case when the action is free then there is a much closer relation between the picard schemes of $X$ and $Y$.

\begin{theorem}~\cite{Je78}
Suppose that $X$ is a proper variety with a free $\alpha_p$ or $\mu_p$-action. Then there exists an exact sequence of group schemes
\[
0 \rightarrow D(N) \rightarrow \underline{Pic}_k(Y) \rightarrow (\underline{Pic}(X)_k)^N \rightarrow 0,
\]
where $N=\alpha_p$ or $\mu_p$, $D(N)$ its Cartier dual and $(\underline{Pic}(X)_k)^N$ is the fixed scheme of $\underline{Pic}(X)_k$ under the natural action of $N$ on $\underline{Pic}(X)_k$.
\end{theorem}

\section{$\alpha_L$-Torsors.}

In this section the notion of $\alpha_L$-torsors is defined. This is a specific class of purely inseparable morphisms of degree $p$ and as it will be shown in the next paragraphs it is very closely related to quotients by $\alpha_p$ or $\mu_p$-actions.

\begin{definition}
Let $G \rightarrow Y$ be a group scheme over a scheme $Y$ acting on a $Y$-scheme $X$. Then $X$ is called a torsor for $G$ over $Y$ if and only if $X$ is faithfully flat over $Y$ and the natural map $G \times_Y X \rightarrow X \times_Y X$ is an isomorphism.   

Equivalently, $X$ is called a torsor for $G$ over $Y$ if an only if for any $Y$ scheme $S$, $\mathrm{Hom}_Y(S,X)$ is a principal homogeneous space over $\mathrm{Hom}_Y(S,G)$.
\end{definition}

Let $X$ be a scheme defined over a field $k$ of characteristic $p>0$ and $L$ a line bundle over $X$. Next I will define the infinitesimal group scheme $\alpha_L$. 

Both $L$ and $L^p$ have a natural structure of group schemes over $X$ and the relative Frobenious map induces a group scheme map $F \colon L \rightarrow L^p$. Over any affine open subset $U$ of $X$ where $L$ trivializes, $F$ is given by $F \colon \mathcal{O}_U[t] \rightarrow \mathcal{O}_U[t]$, where $F(f(t))=f(t^p)$, for any $f \in \mathcal{O}_U[t]$. Moreover, $F$ is surjective in the flat topology over $X$.

\begin{definition}\cite{Ek86}
The group scheme $\alpha_L$ is defined to be the kernel of the relative Frobenious map $F \colon L \rightarrow L^p$. Hence there is an exact sequence of group schemes in the flat topology
\[
0 \rightarrow \alpha_L \rightarrow L \stackrel{F}{\rightarrow} L^p \rightarrow 0.
\]
\end{definition}
Equivalently $\alpha_L$ may be defined as follows. Let $U_i$ be an affine open cover of $X$. Then $L$ is determined by elements $a_{ij} \in \mathcal{O}_{U_{ij}}^{\ast}$ which satisfy the cocycle condition. Then $\alpha_L$ is obtained by glueing the schemes $U_i \times \mathrm{Spec} \left( k[t]/(t^p) \right)$ over $U_{ij}$ via the isomorphisms 
\[
\phi_{ij} \colon U_{ij} \times  \mathrm{Spec} \left( k[t]/(t^p) \right) \rightarrow U_{ji} \times  \mathrm{Spec} \left( k[t]/(t^p) \right)
\]
given at the sheaf of rings level by $\phi_{ij} \colon \mathcal{O}_{U_{ji}}[t]/(t^p) \rightarrow \mathcal{O}_{U_{ij}}[t]/(t^p)$, where $\phi_{ij}(t)=a_{ij}t$. The group scheme structure of $\alpha_L$ is inherited from the one on $L$ and is given locally by the map
\[
\mu_i \colon \mathcal{O}_{U_{i}}[t]/(t^p) \rightarrow \mathcal{O}_{U_{i}}[t]/(t^p) \otimes \mathcal{O}_{U_{i}}[t]/(t^p),
\]
defined by $\mu_i(t)=t \otimes 1 + 1\otimes t$.

From the general theory of torsors~\cite{Mi80}, it follows that $\alpha_L$-torsors over $X$ are classified by $H^1_{\mathrm{fl}}(X, \alpha_L)$. The next theorem describes explicitly the structure of $\alpha_L$-torsors.

\begin{theorem}\label{sec3-th1}
Let $L$ be a line bundle over a scheme $Y$ defined over an algebraically closed field $k$ of characteristic $p>0$.  Let $\pi \colon X \rightarrow Y$ be an $\alpha_L$-torsor over $Y$. Then,
\begin{enumerate}
\item $X$ is completely determined by an $F$-split extension
\begin{gather}\label{sec3-eq1}
0 \rightarrow \mathcal{O}_Y \stackrel{i}{\rightarrow} E \rightarrow L^{-1} \rightarrow 0,
\end{gather}
where $E$ is a rank two vector bundle on $Y$. (An exact sequence is called $F$-split if its pullback by the absolute Frobenious $F$ is split). Moreover, there is a diagram
\[
\xymatrix{
X \ar[r]^{\sigma} \ar[dr]^{\pi} & \mathrm{S(E)} \ar[d] \\
                                &   Y
}
\]
where $\sigma$ is a closed immersion and $N_{X/S(E)}=\pi^{\ast}(\mathcal{O}_Y \oplus L^p)$.
\item Let $U_i$ be an open affine cover of $Y$. Then $\pi \colon X \rightarrow Y$ is obtained by glueing the maps
\[
\pi_i \colon V_i= \mathrm{Spec} \frac{\mathcal{O}_{U_i}[t]}{(t^p-c_i)} \rightarrow U_i,
\]
by the isomorphisms $\psi_{ij} \colon V_{ij} \rightarrow V_{ji}$ defined by sheaf of $\mathcal{O}_{U_{ij}}$-algebras isomorphisms
\[
\phi_{ij} \colon \frac{\mathcal{O}_{U_{ji}}[t]}{(t^p-c_j)} \rightarrow \frac{\mathcal{O}_{U_{ij}}[t]}{(t^p-c_i)}
\]
given by $\phi_{ij}(t)=\gamma_{ij}+a_{ij}t$, where the $a_{ij}\in \mathcal{O}^{\ast}_{U_{ij}}$, $\gamma_{ij} \in  \mathcal{O}_{U_{ij}}$, $c_i \in \mathcal{O}_{U_i}$ and moreover. The elements $a_{ij} \in \mathcal{O}_{U_{ij}}^{\ast}$ satisfy the cocycle condition $a_{ij}a_{jk}=a_{ik}$ and define the line bundle $L^{-1}$. The matrices 
$ A_{ij}= \begin{pmatrix} 1 & \gamma_{ij} \\
0 & a_{ij}
\end{pmatrix} $ satisfy the cocycle condition $A_{ij}A_{jk}=A_{ik}$ and define $E$ in~\ref{sec3-eq1} and $c_j=\gamma_{ij}^p+a_{ij}^pc_i$.
\item Suppose that $Y$ is a projective, Cohen-Macauley scheme and Gorenstein in codimension one. Then $X$ is also Cohen-Macauley, Gorenstein in codimension one and moreover
\[
\omega_X=\pi^{\ast}(\omega_Y \otimes L^{p-1})^{[1]}.
\]
\item If the exact sequence~\ref{sec3-eq1} splits then $X$ admits a nontrivial $\mu_p$-action and $\pi \colon X \rightarrow Y$ is the quotient of $X$ by the $\mu_p$-action. If $H^0(L) \not= 0$, then $X$ admits a nontrivial $\alpha_p$-action and $\pi \colon X \rightarrow Y$ is the quotient of $X$ by the $\alpha_p$-action.
\end{enumerate}
\end{theorem} 

\begin{proof}
The proof of the first part of~\ref{sec3-th1}.1 and~\ref{sec3-th1}.3 can be found in~\cite{Ek86}. For the convenience of the reader and in order to show the remaing parts of the theorem, I will describe the correspondence between $\alpha_L$-torsors and $F$-split exact sequences as in~\ref{sec3-th1}.1.

By~\cite{Mi80}, $\alpha_L$-torsors are classified by $H^1_{\mathrm{fl}}(X, \alpha_L)$. Taking flat cohomology in the exact sequence
\[
0 \rightarrow \alpha_L \rightarrow L \stackrel{F}{\rightarrow} L^p \rightarrow 0.
\]
we get the exact sequence in flat cohomology
\begin{gather}\label{sec3-eq3}
\cdots \rightarrow H^0(Y,L) \rightarrow H_{fl}^1(Y,\alpha_L) \rightarrow H^1(Y,L) \stackrel{F^{\ast}}{\rightarrow} H^1(Y,L)
\end{gather}
where $F^{\ast}$ is the induced map of the Frobenious on cohomology. Now taking into consideration that $H^1(Y,L)=\mathrm{Ext}^1_Y(L^{-1},\mathcal{O}_Y)$ and that the last group classifies extensions of $L^{-1}$ by $\mathcal{O}_Y$, it follows immediately that to give an element of $H^1(Y,\alpha_L)$ is equivalent to give an $F$-split extension 
\begin{gather*}
0 \rightarrow \mathcal{O}_Y \stackrel{i}{\rightarrow} E \rightarrow L^{-1} \rightarrow 0,
\end{gather*}
of $L^{-1}$ by $\mathcal{O}_Y$. Now if one writes down explicitly the correspondence between $\alpha_L$-torsors and elements of $H^1_{\mathrm{fl}}(X, \alpha_L)$ as well as the maps that appear in the exact sequence~\ref{sec3-eq3} it follows that 
$X=\mathrm{Spec}_YR(E,\sigma)$, where $R(E,\sigma)=S(E)/I$,  $I=J+\mathrm{Ker}\Phi$, $J$ is the ideal of $S(E)$ generated by $a-i(a)$, for all $a \in \mathcal{O}_Y$, and $\Phi \colon S(E) \rightarrow \mathcal{O}_Y$ is the map induced on $S(E)$ by $\sigma$, where $\sigma \colon E \rightarrow \mathcal{O}_Y$ is the $F$-splitting. In particular, locally $\Phi$ is defined by setting $\Phi(a)=a^p$, for all $a \in \mathcal{O}_Y=S^0(E)$, and on $S^mE$ by $\Phi(x_1x_1\cdots x_m)=\sigma(x_1)\sigma(x_2)\cdots \sigma(x_m)$. This construction exhibits $X$ as a closed subscheme of $S(E)$, viewed as a vector bundle over $Y$.

The second part of the theorem is obtained by writing down explicitly the above construction in local coordinates.

Next I will show the second part of~\ref{sec3-th1}.1, i.e., that $N_{X/S(E)}=\pi^{\ast}(\mathcal{O}_Y \oplus L^p)$. Let $U_i$ be an affine cover of $Y$. Then $E$ is determined by matrices $ A_{ij}= \begin{pmatrix} 1 & \gamma_{ij} \\
0 & a_{ij}
\end{pmatrix} $ satisfying the usual cocycle condition $A_{ij}A_{jk}=A_{ik}$, where $a_{ij}\in \mathcal{O}_{U_{ij}}^{\ast}$ determine $L^{-1}$ and $\gamma_{ij}\in \mathcal{O}_{U_{ij}}$. Then $S(E)$, viewed as a vector bundle over $Y$,  is obtained by glueing $U_i \times \mathrm{A}_k^2$ along the isomorphisms $\Psi_{ij} \colon U_{ij} \times \mathrm{A}_k^2 \rightarrow U_{ji} \times \mathrm{A}_k^2$ defined by the maps 
\[
\psi_{ij}\colon \mathcal{O}_{U_{ji}}[s,t] \rightarrow \mathcal{O}_{U_{ij}}[s,t]
\]
given by $\psi_{ij}(s)=s$, $\psi_{ij}(t)=\gamma_{ij}s+a_{ij}t$. From the local construction explained in the second part of the theorem it follows that $\pi\colon X \rightarrow Y$ is obtained by glueing over $U_i$ the closed subschemes 
\[
V_i=\mathrm{Spec} \frac{\mathcal{O}_{U_i}[s,t]}{(s-1,t^p-c_i)} 
\]
of $U_i \times \mathbb{A}^2_k $ along the isomorphisms $\Phi_{ij} \colon V_{ij}\rightarrow V_{ji}$ induced from $\Psi_{ij}$. Let $I_i$ be the ideal sheaf of $V_i$ in $U_i \times \mathbb{A}^2_k$. Then $I_i=(s-1,t^p-c_i)$. Then there exists a commutative diagram
\[
\xymatrix{
I_{ji}/I^2_{ji} \ar[r]^{G_{ji}} \ar[d]_{\phi_{ij}} &  \mathcal{O}_{V_{ji}} \oplus \mathcal{O}_{V_{ji}} \ar[d]^{\lambda_{ji}} \\
I_{ij}/I^2_{ij} \ar[r]^{G_{ij}}                    &  \mathcal{O}_{V_{ij}} \oplus \mathcal{O}_{V_{ij}} 
}
\]
such that the maps $\phi_{ij}$ are induced from $\psi_{ij}$ and all the other maps are isomorphisms defined by $G_{ij}(s-1)=(1,0)$, $G_{ij}(t^p-c_j)=(0,1)$ and $\lambda_{ij}(1,0)=(1,0)$, $\lambda_{ij}(0,1)=(0,a_{ij}^p)$. The maps $\lambda_{ij}$ correspond to the matrix $ \begin{pmatrix} 1 & 0 \\
0 & a_{ij}^p
\end{pmatrix} $. Therefore the sheaves $\mathcal{O}_{V_i} \oplus \mathcal{O}_{V_i}$ glue by $\lambda_{ij}$ to $\pi^{\ast}\left( \mathcal{O}_Y \oplus L^{-p}\right) $ and therefore
\[
N_{X/S(E)}=\pi^{\ast}\left(\mathcal{O}_Y \oplus L^p\right),
\]
as claimed. 

Next I will show the adjunction formula in the third part of the theorem. Since $Y$ is Gorenstein in codimension one then by the explicit description of $X$ given in the second part of the theorem it follows that $X$ is Gorenstein in codimension one as well. Hence in order to prove the adjunction formula it suffices to assume that $X$ and $Y$ are both Gorenstein. 

\textbf{Claim:} \[
\omega_X=\pi^{\ast}\omega_Y \otimes \pi^{!}\mathcal{O}_{Y}.
\]
Indeed. Since $\pi$ is finite then $\omega_X=\pi^{!}\omega_Y$~\cite[Chapter III, ex. 7.2]{Ha77}. Then by using duality for finite flat morphisms it follows that
\begin{gather*}
\mathcal{H}om_X(\pi^{\ast}\omega_Y,\omega_X)=\mathcal{H}om_X(\pi^{\ast}\omega_Y,\pi^{!}\omega_Y) =\mathcal{H}om_Y(\pi_{\ast}\pi^{\ast}\omega_Y,\omega_Y) =\\
\mathcal{H}om_Y(\omega_Y\otimes \pi_{\ast}\mathcal{O}_X,\omega_Y)=\mathcal{H}om_Y(\pi_{\ast}\mathcal{O}_X,\mathcal{O}_Y)=\pi^{!}\mathcal{O}_Y
\end{gather*}
and therefore $\omega_X=\pi^{\ast}\omega_Y \otimes \pi^{!}\mathcal{O}_{Y}$, as claimed. Hence in order to prove the adjunction formula in~\ref{sec3-th1}.3 it suffices to show that $\pi^{!}\mathcal{O}_Y=\pi^{\ast}(L^{p-1})$. From the claim and since $X$ and $Y$ are Gorenstein, $\pi^{!}\mathcal{O}_Y$ is an invertible sheaf on $X$. In order to show that it is in fact equal to $\pi^{\ast}(L^{p-1})$ I will describe the local glueing data for $\pi^{!}\mathcal{O}_Y$. Let $U_i$ be an affine open cover of $Y$. Then $V_i=\pi^{-1}U_i$ is an affine open cover of $X$. Moreover,
\[
\mathcal{O}_{V_i}=\frac{\mathcal{O}_{U_i}[t]}{(t^p-c_i)}
\]
where $c_i\in \mathcal{O}_{U_i}$ and the glueing data are as in~\ref{sec3-th1}.2. Then
\[
\pi^{!}\mathcal{O}_Y|_{V_i}=\mathrm{Hom}_{U_i}(\mathcal{O}_{V_i},\mathcal{O}_{U_i})=\mathrm{Hom}_{U_i}(\frac{\mathcal{O}_{U_i}[t]}{(t^p-c_i)},\mathcal{O}_{U_i}).
\]
Then I claim that $\phi_i=d^{p-1}/dt^{p-1}$ is a generator of $\pi^{!}\mathcal{O}_Y|_{V_i}$ as an $\mathcal{O}_{U_i}$-module. Indeed. Let $g_i$ be a generator. Then $d^{p-1}/dt^{p-1}=ag_i$, for some $a \in \mathcal{O}_{V_i}$. Then
\[
(p-1)!=\frac{d^{p-1}}{dt^{p-1}}(t^{p-1})=ag_i(t^{p-1}).\]
But $(p-1)!=-1 \mathrm{mod} p$ and hence $a \in \mathcal{O}_{V_i}^{\ast}$. Hence $\phi_i=d^{p-1}/dt^{p-1}$ is a generator of $\pi^{!}\mathcal{O}_Y|_{V_i}$, as claimed. Now the $V_i$ glue in order to form $X$ by the isomorphisms  $\psi_{ij} \colon V_{ij} \rightarrow V_{ji}$ defined by sheaf of $\mathcal{O}_{U_{ij}}$-algebras isomorphisms
\[
\phi_{ij} \colon \mathcal{O}_{V_{ji}}=\frac{\mathcal{O}_{U_{ji}}[t]}{(t^p-c_j)} \rightarrow \frac{\mathcal{O}_{U_{ij}}[t]}{(t^p-c_i)}=\mathcal{O}_{V_{ij}}
\]
given by $\phi_{ij}(t)=\gamma_{ij}+a_{ij}t$, where $a_{ij}$ and $\gamma_{ij}$ are as in~\ref{sec3-th1}.2. Then $\mathrm{Hom}_{U_i}(\mathcal{O}_{V_i},\mathcal{O}_{U_i})$ glue to form $\pi^{!}\mathcal{O}_Y$ by the isomorphisms
\[
\delta_{ij} \colon \mathrm{Hom}_{U_{ij}}(\mathcal{O}_{V_{ij}},\mathcal{O}_{U_{ij}}) \rightarrow \mathrm{Hom}_{U_{ji}}(\mathcal{O}_{V_{ji}},\mathcal{O}_{U_{ji}})
\]
given by $\delta_{ij}(\sigma)=\sigma \circ \phi_{ij}$, for any $\sigma \in \mathrm{Hom}_{U_{ij}}(\mathcal{O}_{V_{ij}},\mathcal{O}_{U_{ij}})$. In particular,
\[
\delta_{ij}(\frac{d^{p-1}}{dt^{p-1}})=\frac{d^{p-1}}{dt^{p-1}}\circ \phi_{ij}\]
But for any $f(t) \in \mathcal{O}_{U_{ji}}[t]$,
\begin{gather*}
\frac{d^{p-1}}{dt^{p-1}}\circ \phi_{ij}(f(t))=\frac{d^{p-1}}{dt^{p-1}}(f(\gamma_{ij}+a_{ij}t))=a_{ij}^{p-1}\frac{d^{p-1}f}{dt^{p-1}}(\gamma_{ij}+a_{ij}t)=\\
a_{ij}^{p-1}\frac{d^{p-1}f}{dt^{p-1}}=a_{ji}^{1-p}\frac{d^{p-1}f}{dt^{p-1}}.
\end{gather*}
Hence the glueing data for $\pi^{!}\mathcal{O}_Y$ are given by $a_{ij}^{1-p}$ and therefore $\pi^{!}\mathcal{O}_Y=\pi^{\ast}(L^{p-1})$, as claimed.

Next I will show the last part of the theorem.

Suppose that the exact sequence~\ref{sec3-eq1} is split. Then $E=\mathcal{O}_Y \oplus L^{-1}$. Then in the notation of~\ref{sec3-th1}.2, $\gamma_{ij}=0$. Let $D_i$ be the $\mathcal{O}_{U_i}$-derivation of $\mathcal{O}_{V_i}$ given by $D_i=t\frac{d}{dt}$. Then it is easy to see that $D_i^p=D_i$ and that there exists a commutative diagram
\[
\xymatrix{
\frac{\mathcal{O}_{U_{ji}}[t]}{(t^p-c_j)} \ar[d]_{D_j}\ar[r]^{\phi_{ij}} & \frac{\mathcal{O}_{U_{ij}}[t]}{(t^p-c_i)} \ar[d]^{D_i} \\
\frac{\mathcal{O}_{U_{ji}}[t]}{(t^p-c_j)} \ar[r]^{\phi_{ij}} & \frac{\mathcal{O}_{U_{ij}}[t]}{(t^p-c_i)}
}
\] 
Therefore the derivations $D_i$ glue to a global derivation $D$ on $X$ such that $D^p=D$. Then by Proposition~\ref{sec1-prop1} $D$ induces a $\mu_p$-action on $X$ and $Y$ is the quotient of $X$ by this action.

Suppose that $H^0(Y,L)\not=0$. Then there exists a nonzero map $\sigma \colon L^{-1} \rightarrow \mathcal{O}_Y$. Suppose $U_i$ is an open affine cover of $Y$ and that $a_{ij} \in \mathcal{O}_{U_{ij}}^{\ast}$ are the glueing data of $L^{-1}$. Then $\sigma$ is obtained by glueing maps $\sigma_i \colon \mathcal{O}_{U_i} \rightarrow \mathcal{O}_{U_i}$ that are compatible with the glueing data of $L^{-1}$, i.e., there are commutative diagrams
\[
\xymatrix{
\mathcal{O}_{U_{ij}}\ar[d]_{\lambda_{ij}}\ar[r]^{\sigma_i} & \mathcal{O}_{U_{ij}}\\
\mathcal{O}_{U_{ji}}\ar[ur]^{\sigma_j}
}
\]
where $\lambda_{ij}(x)=a_{ji}x$, for any $x \in \mathcal{O}_{U_{ij}}$. Suppose that $d_i=\sigma_i(1)$. Then $d_i=a_{ji}d_j$. Let now $D_i=d_i\frac{d}{dt}$. These are $\mathcal{O}_{U_i}$-derivations of $\mathcal{O}_{V_i}$ such that $D_i^p=0$. An argument identical to the one used in the previous case shows that the $D_i$ glue to a global derivation $D$ on $X$ such that $D^p=0$. By Proposition~\ref{sec1-prop1}, $D$ induces an action of $\alpha_p$ on $X$ and $Y$ is the quotient of $X$ by this action.

\begin{remark}
The statement of Theorem~\ref{sec3-th1}.1 was initially proved in~\cite{Ek86} as well as~\ref{sec3-th1}.3 in the case when $Y$ is normal. The rest are to the best of my knowledge new.
\end{remark}

\end{proof}

\section{Quotients by $\mu_p$-actions.}\label{mu-p}
Let $X$ be a scheme defined over a field $k$ of characteristic $p>0$. Suppose that $X$ admits a nontrivial $\mu_p$-action. By Proposition~\ref{sec1-prop1} this action is induced by a global vector field $D$ of $X$ such that $D^p=D$. let $\pi \colon X \rightarrow Y$ be the quotient, which exists and is finite of dgree $p$~\cite{Mu70}. The purpose of this section is to describe the structure of the quotient map $\pi$.

\begin{lemma}\label{sec4-lemma1}
Let $k$ be a field of characteristic $p>0$. Let $A$ be a $k$-algebra and $M$ an $A$-module. Let $\Phi \colon M \rightarrow  M$ be an $A$-module homomorphism such that $\Phi^p=\Phi$. Then $M =\oplus_{k=0}^{p-1}M_k$, where $M_k=\{m \in M| \Phi(m)=km\}$.
\end{lemma}

\begin{proof}
Let $f(x)=x^p-x \in \mathbb{Z}_p[x]$. Then $f(x)=\prod_{k=0}^{p-1}(x-k)$. From the assumption it follows that  $f(\Phi)=0$. Let $f_{k}(x)=-\prod_{i \not= k}(x-i)=-(x^p-x)/(x-k)$. I will show that 
\begin{gather}\label{sec4-eq1}
\sum_{k=0}^{p-1}f_k(x)=1.
\end{gather}
In order to show this I will first show that  $\sum_{k=0}^{p-1}f^{\prime}_k(x)=0$. Indeed,
\[
f_k^{\prime}(x)=\frac{x^p-k}{(x-k)^2}=\frac{(x-k)^p}{(x-k)^2}=(x-k)^{p-2}.
\]
Then \[
g(x)=\sum_{k=0}^{p-1}f^{\prime}_k(x)=\sum_{k=0}^{p-1}(x-k)^{p-2}
\]
is a polynomial in $\mathbb{Z}_p[x]$ of degree at most $p-2$. However, for any $a \in \mathbb{Z}_p$,
\[
g(a)=\sum_{k=0}^{p-1}(a-k)^{p-2}=\sum_{s=0}^{p-1}s^{p-2}=\sum_{s=1}^{p-1}s^{-1}=\sum_{s=1}^{p-1}s=\frac{p(p-1)}{2}=0.
\]
Hence $g(x)$ has $p$-roots in $\mathbb{Z}_p$ and since its degree is at most $p-2$ it follows that $g(x)=0$. Hence 
\[
\left(\sum_{k=0}^{p-1}f_k(x)\right)^{\prime}=0
\]
and therefore $\sum_{k=0}^{p-1}f_k(x)=h(x^p)$, for some $h(x)\in \mathbb{Z}_p[x]$. However, since the degree of $\sum_{k=0}^{p-1}f_k(x)$ is at most $p-1$, it follows that $h(x)=c, c\in \mathbb{Z}_p$. Then \[
c=h(0)=\sum_{k=0}^{p-1}f_k(0)=f_0(0)=-(p-1)!=1,
\]
from Wilson's theorem. Therefore~\ref{sec4-eq1} holds. Let $m \in M$. Then from~\ref{sec4-eq1} it follows that
\[
m=\sum_{k=0}^{p-1}f_k(\Phi)(m).
\]
Now from the definition of $f_k(x)$, $(x-k)f_k(x)=x^p-x$ and therefore $\Phi(f_k(\Phi)(m))=kf_k(\Phi)(m)$. Hence $f_k(\Phi)(m) \in M_k$ and therefore \[
M=\sum_{k=1}^{p-1}M_k.
\]
Next I will show that the sum is direct. Suppose that there are $m_k \in M_k$, $k=0,1,\ldots,p-1$ such that 
\[
m_0+m_1+\cdots + m_{p-1}=0.
\]
Applying $\Phi$ $(p-1)$-times we get that
\[
\sum_{k=1}^{p-1}k^sm_k = 0,
\]
for $s=1,2, \ldots, p-1$. This equation can be written in matrix form as 
\[
A \cdot B = \mathbf{0},
\]
where 

\begin{tabular}{lr}
$ A=\begin{pmatrix}
1 & 2 & \cdots & p-1 \\
1 & 2^2 & \cdots & (p-1)^2 \\
\hdotsfor{4} \\
1 & 2^{p-1} & \cdots & (p-1)^{p-1} \\
\end{pmatrix}, $ 
&
$ B= \begin{pmatrix}
m_1 \\
m_2 \\
\vdots \\
m_{p-1}
\end{pmatrix}
$
\end{tabular}

But $A$ is simply a Vandermonde matrix and hence 
\[
\det(A)=\prod_{1 \leq i < j \leq p-1}(i-j) \not= 0.
\]
Therefore $B=0$ and hence $m_k=0$, for all $k=0,1,\ldots, p-1$. Therefore 
\[
M=\oplus_{k=0}^{p-1}M_k,
\]
as claimed.
\end{proof}

The next result shows the relation between $L_k(D)$ and $L_{kd}(D)$ (as defined in~\ref{def-of-L-N}) in the case when $D^p=D$. This is very important in the study of the quotient of a scheme with a $\mu_p$-action.

\begin{proposition}\label{mu-p-decomposition}
Let $A$ be an intergral domain which is also a $k$-algebra where $k$ is a field of characteristic $p>0$. Let $D\in \mathrm{Der}_k(A)$ a non-trivial $k$-derivation of $A$ such that $D^p=D$. Then 
\begin{enumerate}
\item There is a direct sum decomposition of $B=A^D$-modules 
\begin{gather}\label{sec4-eq2}
A = \oplus_{k=0}^{p-1}L_k(D).
\end{gather}
Moreover, $L_k(D)$ are rank 1 torsion free $B$-modules.
\item For any $k \in \{1,2,\ldots,p-1\}$, let 
\[
\sigma_d \colon L^{\otimes d}_k(D) \rightarrow L_{\overline{kd}}(D)
\]
where $\overline{kd}=kd \mod p$, be the map defined by $\sigma_d(a_1\otimes a_2 \otimes \cdots \otimes a_d)= a_1a_2\cdots a_d$. Then the support of the kernel and cokernel of $\sigma_d$ is contained in the fixed locus of $D$. In particular, if the fixed locus of $D$ is empty then $\sigma_d$ is an isomorphism.
\end{enumerate}
\end{proposition}

\begin{proof}

Let $\Phi \colon A \rightarrow A$ be the map defined by $\Phi(a)=Da$, for any $a \in A$. Then $\Phi^p=\Phi$. Hence from Proposition~\ref{sec4-lemma1} it follows that
\begin{equation}\label{direct-sum}
A = \oplus_{k=0}^{p-1}L_k(D).
\end{equation}
Since $L_k(D) \subset A$, $L_k(D)$ is a torsion free $B$-module. 

I will next show that $L_k(D) \not=0$, for any $k \in \{0,1,\ldots, p-1\}$. For any $a \in A$, let $z_a=Da+D^2a+\cdots + D^{p-1}a$. Then $Dz_a=z_a$. I claim that there is $a\in A$ such that $z_a \not= 0$. Suppose that this was not true. Then $D, D^2, \ldots, D^{p-1}$ would be linearly dependent over the base field $k$. But this is impossible~\cite[Theorem 25.4]{Ma86}. Let then $a \in A$ such that $z_a \not=0$. Then for any $k=1,2,\ldots,p-1$, $D(z_a^k)=kz_a^{k-1}Dz_a=kz_a^k$. Hence $z_a^k \in L_k(D)$. Therefore $L_k(D) \not= 0$, for any $k=0,1,\ldots,p-1$. Now localizing~\ref{direct-sum} at the generic point of $A$ and taking into consideration that $[K(A):K(B)]=p$ and that by Proposition~\ref{subsec2-prop3} the formation of the modules $L_k(D)$ commutes with localization, we get that $L_k(D)$ has rank 1 for all $k$. 

The second part of the proposition will follow if we show that the map $\sigma_d$ is an isomorphism at any point $P \in \mathrm{Spec}(A)$ that is not fixed by $D$. This is a local result and hence by Proposition~\ref{subsec2-prop3} we may assume that $A$ and $B=A^D$ are local rings with maximal ideals $m_A$ and $m_B$ and that $A$ is a free $B$-module of rank $p$~\cite[Theorem 1, page 104]{Mu70}. For simplicity I will only do the case when $k=1$. The rest are similar.

So suppose that $k=1$. I will first show that $D(L_1(D)) \not\subset m_A$. Suppose on the contrary that $D(L_1(D))\subset m_A$. Then I will show that $D(L_k(D)) \subset m_A$ for any $k$.  Let $a \in L_k$. Then there exists $\nu \in \mathbb{Z}$ such that $\nu k =1 \mod p$. Then $D(a^{\nu})=\nu k a^{\nu}=a^{\nu}$ and hence $a^{\nu} \in L_1$. But then $a^{\nu}=D(a^{\nu})\in m_A$, and therefore $a \in  m_A$. Hence $D(L_k(D)) \subset m_A$, for any $k \in \{1,2,\ldots,p-1\}$. Now from~\ref{direct-sum} it follows that $D(A) \subset m_A$ and hence $m_A$ is fixed under $D$, which is a contradiction. 

Hence $D(L_1(D)) \not\subset m_A$. Therefore there exists $a \in L_1(D) =D(L_1(D))$ such that $a=Da\not \in m_A$. Then $a^k \in L_k(D)$. $L_k(D)$ is a free $B$-module of rank 1. Let $x_k \in L_k(D)$ be a generator as a $B$-module. Then there exists $\lambda \in B$ such that $a^k=\lambda x_k$. If $\lambda \in m_B$, then $a^k \in m_A$ and hence $a \in m_A$, a contradiction. Hence $\lambda \not\in m_B$ and therefore $\lambda \in B^{\ast}$. Hence $a^k$ is also a generator of $L_k(D)$ as a $B$-module. This shows that the map
\[
\sigma_k \colon L_1(D)^{\otimes k} \rightarrow L_k(D)
\]
is surjective and hence an isomorphism.

\end{proof}

\begin{corollary}\label{sec4-cor1}
Let $X$ be an integral scheme defined over a field $k$ of characteristic $p>0$. Suppose that $X$ admits a nontrivial $\mu_p$-action induced by a nontrivial global vector field $D$ such that $D^p=D$. Let $\pi \colon X \rightarrow Y$ be the quotient. Then
\begin{enumerate}
\item There is a direct sum decomposition of $\mathcal{O}_Y$-modules
 \[
\pi_{\ast}\mathcal{O}_X=\oplus_{k=1}^{p-1}L_k(D),
\]
where $L_k(D)=\{a \in \mathcal{O}_X / \; Da=ka\}$ is a rank one torsion free sheaf of $\mathcal{O}_Y$-modules. 
\item For any $k \in \{1,2,\ldots, p-1\}$, $d \in \mathbb{N}$, let 
\[
\sigma_d \colon L_k(D)^{d} \rightarrow L_{\overline{kd}}(D),
\]
where $\overline{kd}=kd \mod p$, be the map defined locally by 
\[
\sigma_d(a_1\otimes a_2 \otimes \cdots \otimes a_d)= a_1a_2\cdots a_d.
\]
Then the support of the kernel and cokernel of $\sigma_d$ is in the fixed locus of the $\mu_p$-action. In particular, if $\mu_p$ acts freely on $X$ then $\sigma_d$ is an isomorphism. In particular, $L_1(D)^{\otimes p} \cong \mathcal{O}_Y$.
\item If  $X$ is $S_2$, then $L_k(D)$ is $S_2$ as well. Moreover, if $X$ is normal and Noetherian, then $L_k(D)$ is a reflexive $\mathcal{O}_Y$-module, for all $k=0,1,\ldots,p-1$. 
\end{enumerate}
\end{corollary}

\begin{proof}
By Proposition~\ref{sec1-prop1}, the $\mu_p$-action on $X$ is induced by a nontrivial global vector field $D$ on $X$ such that $D^p=D$. Then $D$ induces an $\mathcal{O}_Y$-sheaf of modules map $\Phi \colon \mathcal{O}_X \rightarrow \mathcal{O}_X$, by $\Phi(a)=Da$, such that $\Phi^p=\Phi$. Now the parts 1 and 2 of the corollary follow immediately from Proposition~\ref{mu-p-decomposition}, equation~\ref{sec4-eq2}. 

It remains to prove the third part. Suppose that $X$ is $S_2$. Then for any $Z \subset Y$ such that $\mathrm{codim}(Y-Z,Y) \geq 2$, $\underline{\mathrm{H}}^i_Z(\pi_{\ast}\mathcal{O}_X)=0$, for $i=0,1$. Hence from the first part of the corollary it follows that \[
\oplus_{k=0}^{p-1}\underline{\mathrm{H}}^i_Y(L_k(D))=0,
\]
for $i=0,1$. Hence $\underline{\mathrm{H}}^i_Y(L_k(D))=0$, $i=0,1$, and therefore $L_k(D)$ is an $S_2$ sheaf. If in addition $X$ is normal, then $L_k(D)$ is reflexive since the notions of reflexive and $S_2$ sheaves coincide for normal integral schemes~\cite[Theorem 1.9]{Ha94}.
\end{proof}

\begin{remark}
By Proposition~\ref{subsec2-prop3}, the formation of the sheaves $L_k(D)$ is stable by passing to completion and localization. Therefore it is possible to calculate their local properties by passing to either the completion of $X$ at a point or the localization.
\end{remark}

\begin{theorem}\label{sec4-th1}
Let $X$ be a normal integral scheme of finite type over a perfect field $k$ of characteristic $p>0$. Suppose that $X$ admits a nontrivial $\mu_p$-action. Let $\pi \colon X \rightarrow Y$ be the quotient. Then there exists a factorization
\[
\xymatrix{
X \ar[r]^{\phi}\ar[d]_{\pi} & Z \ar[dl]^{\delta} \\
Y   & \\
}
\]
such that, \[
\delta \colon Z=\mathrm{Spec}_Y\left( \oplus_{k=0}^{p-1} L^{[k]}\right) \rightarrow Y, 
\]
is the $p$-cyclic cover over $Y$ defined by a rank one reflexive sheaf $L$ of $Y$ and a section $s \in H^0(Y,L^{[-p]})$, and $\phi \colon X \rightarrow Z$ is the normalization of $Z$. Moreover, 
\begin{enumerate}
\item Suppose that $X$ is smooth and let $Z$ be a connected component of the fixed locus of the $\mu_p$-action. Then $L$ and the section $s \in H^0(Y,L^{[-p]})$ can be chosen in such a way so that there exists an open subset $U$ of $X$ containing $Z$ such that the restriction of $\phi$ on $U$ is an isomorphism. 
\item If $X$ is projective over $k$, then  there exists a rank one reflexive sheaf on $Y$ such that \[
\omega_Z =\left(\pi^{\ast}(\omega_Y \otimes M^{[1-p]})\right)^{[1]}.
\]
\item If either $p=2$ or $\mu_p$ acts freely on $X$, then $\phi$ is an isomorphism for a a suitable choice of $L$ and $s \in H^0(L^{[-p]})$. If the $\mu_p$-action is free then $L^p \cong \mathcal{O}_Y$. 
\end{enumerate}

\end{theorem}

\begin{proof}
Since $X$ is normal, it satisfies condition $S_2$. Hence, $\pi_{\ast}\mathcal{O}_X$ is an $\mathcal{O}_Y$-module that satisfies condition $S_2$. Since $X$ is normal, $Y$ is also normal and therefore $\pi_{\ast}\mathcal{O}_X$ is a reflexive $\mathcal{O}_Y$-module~\cite{Ha94}. Moreover, by Proposition~\ref{sec1-prop1} the $\mu_p$-action is induced by a global vector field $D$ on $X$ such that $D^p=D$. By Corollary~\ref{sec4-cor1},
\[
\pi_{\ast}\mathcal{O}_X=\oplus_{k=0}^{p-1}L_k(D).
\]
Let $k \in \{1,2,\ldots,p-1\}$ and $L=L_k(D)$. Then there exist natural $\mathcal{O}_Y$-module maps \[
\sigma_d \colon L^{\otimes d} \rightarrow L_{\overline{kd}}(D) \subset \pi_{\ast}(\mathcal{O}_X),
\]
where $\overline{kd}=kd \mod p$, given locally by $\sigma_d(a_1\otimes \cdots \otimes a_d)=a_1\cdots a_d \in L_{\overline{kd}}(D)$. Note that by Proposition~\ref{mu-p-decomposition}, the support of the kernel and cokernel of $\sigma_d$ is contained in the fixed locus of the $\mu_p$-action.

Let $U \subset Y$ be the set of points of $Y$ where $\pi_{\ast}\mathcal{O}_X$ is free. Since $\pi_{\ast}\mathcal{O}_X$ is reflexive, $U$ contains the smooth locus of $Y$ and since $k$ is perfect and $X$ of finite type over $k$, $\mathrm{codim}(Y-U,Y)\geq 2$. Hence the restriction $\pi_{\ast}\mathcal{O}_X|_U$ is locally free and of rank $p$, since $\pi$ has degree $p$. Hence 
$L_k(D)|_U$ is invertible for $k=1,2,\ldots,p-1$. Hence the maps $\sigma_d$ are injective over a codimension 2 open subset of $Y$ and hence, after taking double duals, the maps
\[
\sigma_d \colon L^{[d]} \rightarrow \pi_{\ast}(\mathcal{O}_X)
\]
are also injective giving an injection of sheaves of abelian groups 
\[
\oplus_{d=0}^{p-1}L^{[d]} \subset \pi_{\ast}\mathcal{O}_X.
\]
Now observe that the map $\sigma_p \colon L^{\otimes p} \rightarrow \pi_{\ast}\mathcal{O}_X$ given by $\sigma_p(a_1 \otimes \cdots \otimes a_p)=a_1\cdots a_p$ is a map of $\mathcal{O}_Y$ modules and moreover the image of $\sigma_p$ is in $\mathcal{O}_Y$ and hence gives a map $L^p \rightarrow \mathcal{O}_Y$ which induces a map $L^{[p]} \rightarrow \mathcal{O}_Y$. This map corresponds to a section of $L^{[-p]}$ and gives a sheaf of rings structure on $ \oplus_{d=0}^{p-1}L^{[d]}$ compatible with the sheaf of rings structures of $\mathcal{O}_Y$ and $\mathcal{O}_X$. Hence there exists inclusions of sheaves of $\mathcal{O}_Y$-algebras
\[
\mathcal{O}_Y \subset  \oplus_{d=0}^{p-1}L^{[d]} \subset \pi_{\ast}\mathcal{O}_X.
\]
Let $Z=\mathrm{Spec}\left(  \oplus_{d=0}^{p-1}L^{[d]} \right)$. Then there exists a factorization 
\begin{gather}\label{sec4-diagram-1}
\xymatrix{
X \ar[r]^{\phi}\ar[d]_{\pi} & Z \ar[dl]^{\delta} \\
Y   & \\
}
\end{gather}
as claimed in the first part of the theorem. Over $U$, $\delta$ is an $\alpha_L$-torsor and in particular it is finite of degree $p$. But since $\pi$ is also of degree $p$ it follows that $\phi$ is of degree zero and hence it is birational. Since $X$ is normal it follows that $X$ is the normalization of $Z$. Finally I would like to mention that $L=L_1(D)$ would be the standard choice of $L$.

Suppose that $X$ is smooth and let $Z\subset X$ be a connected component of the fixed locus of $D$.  I will show that $Z$ is smooth and irreducible. Let $P \in Z$ be any closed point. Then by~\cite{R-S76}, in suitable local coordinates at $P \in X$, $x_1, \ldots, x_n$, $n =\dim X$, 
\begin{gather}\label{sec4-eq3}
D=a_1x_1\partial /\partial x_1 + \cdots + a_nx_n \partial /\partial x_n,
\end{gather}
where $a_i \in \mathbb{Z}_p$, $i=1,\ldots, n$. Then the fixed locus of $D$ is defined by the ideal $(a_1x_1,\ldots,a_nx_n)$ and therefore it is smooth and irreducible. 

Let $U=X-\cup_{W\not= Z}W$, where $W$ runs over all connected components of the fixed locus of $D$, or equivalently of the $\mu_p$-action.

Suppose that $\mathrm{codim}(Z,X) \geq 2$. Then $U$ is an open neighborhood of $Z$ in $X$ such that the $\mu_p$-action on $U-Z$ is free and therefore $U-Z \rightarrow Y-\pi(Z)$ is a $\mu_p$-torsor. Let $L=L_1(D)$. Then the restriction of $\phi$ on $U$ is an isomorphism in codimension 2 and hence since it is finite it is an isomorphism.  

Suppose that $\mathrm{codim}(Z,X)=1$. By Proposition~\ref{prop3}, the singular points of $Y$ are under the isolated fixed points of the $\mu_p$-action and therefore $V=\pi(U)$ is a smooth open neighborhood of $\pi(Z)$ in $Y$. Hence we may assume that $Y$ is also smooth. From~\ref{sec4-eq3} it follows that $D=ax_1\partial /\partial x_1$, $a\in \mathbb{Z}^{\ast}_p$. Let $\tilde{D}=x_1\partial /\partial x_1$. Then locally the quotient of $X$ by the $\mu_p$-action induced by $\tilde{D}$ is the same as the one induced by $D$. In fact, locally at $P\in X$, $\mathcal{O}_X=k[[x_1,x_2,\ldots,x_n]]$ and $\mathcal{O}_Y=k[[x_1^p,x_2,\ldots,x_n]]$.

\textbf{Claim:} For any $k \in \{1,2,\ldots, p-1\}$, the natural map
\begin{gather}\label{sec4-eq4}
\sigma_k \colon L_1(\tilde{D})^{\otimes k} \rightarrow L_{k}(\tilde{D})
\end{gather}
is surjective and therefore it is an isomorphism since both $L_1(\tilde{D})$ and $L_{k}(\tilde{D})$ are line bundles on $U$.

I proceed to prove the claim. Let $f \in L_k(\tilde{D})$. Then by the definition of $L_k(\tilde{D})$, \[
x_1\frac{\partial f}{\partial x_1}=kf.
\]
Now $f$ can be written as
\[
f(x_1,x_2,\ldots,x_n)=\sum_{d}g_d(x_2,\ldots,x_n)x_1^d.
\]
Then $f \in L_k(\tilde{D})$ if and only if $(d-k)f_d(x_2,\ldots,x_n)=0$, for any $d$, and hence $d=k +\nu p$, $\nu \in \mathbb{Z}$. Hence 
\[
f(x_1,x_2, \ldots, x_n)=\sum_{\nu}g_{\nu p +k}(x_2,\ldots,x_n)x_1^{\nu p +k}.
\]
Therefore, $L_k(\tilde{D})=\mathcal{O}_Y \cdot x_1^k$. Hence it now immediately follows that $\sigma_k$ in~\ref{sec4-eq4} is surjective. This concludes the proof of the claim.

Now observe that \[
L_d(\tilde{D})=L_{\overline{ad}}(D),\]
where $\overline{ad}=ad \mod p$. Hence~\ref{sec4-eq4} is equivalent to say that the natural map
\[
\sigma_k \colon L_a(D)^{\otimes k} \rightarrow L_{\overline{ak}}(D)
\]
is surjective at $P \in Z$ and hence an isomorphism for any $k \in \{1,2,\ldots, p-1\}$. Now the cokernel of $\sigma_k$ is supported on the divisorial part of the fixed locus of $D$. However the connected components of the divisorial part of the fixed locus of $D$ are smooth and irreducible and $\sigma_k$ is an isomorphism at $P \in Z$. Therefore $Z$ is not contained in the support of the cokernel of $\sigma_k$ and hence the restriction of $\sigma_k$ in $U$ is an isomorphism. Moreover, if $k$ runs over all members of $\{1,2,\ldots,p-1\}$, $\overline{ak}$ also runs over all members of $\{1,2,\ldots,p-1\}$. Hence since $\sigma_k$ is an isomorphism in $U$, it follows that
\[
\pi_{\ast}\mathcal{O}|_V = \oplus_{k=0}^{p-1}L_k(D)|_V=\oplus_{k=0}^{p-1}L_{\overline{ak}}(D)|_V=\oplus_{k=0}^{p-1}L_a(D)^{k}|_V,
\]
where $V=\pi(U)$. Take now $L=L_a(D)$ in the construction of diagram~\ref{sec4-diagram-1}. Then $\phi$ is an isomorphism in $U\supset Z$. Then by Theorem~\ref{sec3-th1}, 
\[
\omega_U=\pi^{\ast}(\omega_V \otimes L_a^{1-p}).
\]
Moreover, $\mathcal{O}_X(Z)=(\pi^{\ast}L^{-1}_a)^{[1]}
$.

Suppose that $X$ is projective over $k$. Then $Y$ is projective as well and therefore it has a dualizing sheaf $\omega_Y$. Let $Z$ be a connected component of the divisorial locus of $D$. Let $U_Z=X-\cup_{W\not= Z}$, where $W$ is a connected component of the fixed locus of $D$ different than $Z$, and let $V_Z=\pi(U_Z)$. Then from the previous discusion, there exists a $k_z \in \{1,2,\ldots,p-1\}$, such that \[
\omega_{U_Z}=\pi^{\ast}(\omega_{V_Z}\otimes L_{k_z}(D)^{1-p}).
\]
Let $M=\left(\otimes_Z L_{k_z}(D)\right)^{[1]}$, where $Z$ runs over all connected components of the divisorial part of the fixed locus of $D$. Then 
\[
\omega_Z =\left(\pi^{\ast}(\omega_Y \otimes M^{[1-p]})\right)^{[1]}.
\]

Finally suppose that $p=2$ or that $\mu_p$ acts freely on $X$. If $p=2$, then from Proposition~\ref{sec4-cor1},
\[
\pi_{\ast}\mathcal{O}_X=\mathcal{O}_Y \oplus L,
\]
where $L=L_1(D)$ and hence $X=Z$.

If $\mu_p$ acts freely on $X$ then $Y$ is smooth and $L_k(D)$ is invertible for all $k$. Moreover, by Corollary~\ref{sec4-cor1}, $L_k(D) \cong L_1(D)^{\otimes k}$, for any $k \in \{0,1,2, \ldots ,p-1\}$. Hence 
\[
\pi_{\ast}\mathcal{O}_X=\oplus_{k=0}^{p-1}L_k(D) =\oplus_{k=0}^{p-1} L_1(D)^{\otimes k}
\]
and hence $X=Z$ for $L=L_1(D)$.
\end{proof}
\begin{remark}
The third part of Theorem~\ref{sec4-th1} is not new. The statement about the free $\mu_p$ action is in~\cite[Page 125]{Mi80}, and the statement about the characteristic 2 case in~\cite{Ek86}. However the exposition here is completely self contained.
\end{remark}

\section{Quotients by $\alpha_p$-actions.}\label{alpha-p}
Let $X$ be a scheme defined over a field $k$ of characteristic $p>0$. Suppose that $X$ admits a nontrivial $\alpha_p$-action. By Proposition~\ref{sec1-prop1} this action is induced by a global vector field $D$ of $X$ such that $D^p=0$. Let $\pi \colon X \rightarrow Y$ be the quotient, which exists and is finite of degree $p$~\cite{Mu70}. The purpose of this section is to describe the structure of the quotient map $\pi$.

\begin{proposition}\label{section5-prop1}
Let $X$ be a normal scheme defined defined over a field $k$ of characteristic $p>0$. Suppose that $X$ admits a nontrivial $\alpha_p$-action induced by a nontrivial vector field $D$ of $X$ such that $D^p=0$. Then there exists a filtration of $\mathcal{O}_Y$-modules
\[
\mathcal{O}_Y=E_0\subsetneqq E_1 \subsetneqq E_2 \subsetneqq \cdots      \subsetneqq E_{k}  \subsetneqq E_{k+1}   \subsetneqq \cdots   \subsetneqq E_{p-2}   \subsetneqq E_{p-1} =\pi_{\ast}\mathcal{O}_X,
\]
where \[
E_k=\{a \in \mathcal{O}_X |\; D^{k+1}a=0\}
\]
is a reflexive sheaf of $\mathcal{O}_Y$-modules of rank $k+1$ and the quotients $L_k=E_k/E_{k-1}$ are torsion free sheaves of rank 1. Moreover, $E_k$ and $L_k$ are locally free outside the fixed locus of the $\alpha_p$ action, for all $k$. In addition,
\begin{enumerate}
\item The natural maps $\tau_k \colon E_1^{[k]}\rightarrow E_k$ defined by $\tau_k(a_1\otimes \cdots \otimes a_k)=a_1\cdots a_k$ induce $\mathcal{O}_Y$-module maps  \[
\sigma_k \colon L_1^{ [k]} \rightarrow L_k^{[1]}
\]
which are isomorphisms outside the fixed locus of the $\alpha_p$-action.
\item $H^0(Y,L_1^{[-1]})\not=0$.
\item The map $\lambda \colon E_1 \rightarrow \mathcal{O}_Y$ given by $\lambda(a)=a^p$ is an $F$-spliting of the exact sequence
\[
0 \rightarrow \mathcal{O}_Y \rightarrow E_1 \rightarrow L_1 \rightarrow 0.
\]

\end{enumerate}

\end{proposition} 

\begin{proof}
Let $E_k=\{a \in \pi_{\ast}\mathcal{O}_X| \; D^{k+1}a=0 \}$. Then $E_k$ is clearly a sheaf of $\mathcal{O}_Y$-modules and  $E_k \subset E_{k+1}$. I will first show that $E_k\subsetneqq E_{k+1}$, for all $k=0,1,2,\ldots,p-2$. Suppose on the contrary that there is $k \in\{0,1,2,\ldots,p-2\}$ such that $E_k=E_{k+1}$. Let $a \in \pi_{\ast}\mathcal{O}_X$. Then $D^{p-k-2}a \in E_{k+1}$. But since $E_{k+1}=E_k$, it follows that $D^{p-k-2}a \in E_k$ and therefore $D^{k+1}(D^{p-k-2}a)=0$, and hence $D^{p-1}a=0$, for all $a \in \pi_{\ast}\mathcal{O}_X$. Hence $D^{p-1}$ is a derivation, which is impossible~\cite[Theorem 25.4]{Ma86}. Therefore $E_k \subsetneqq E_{k+1}$, for all $k$ and hence there exists a filtration
\begin{gather}\label{filtration}
\mathcal{O}_Y=E_0\subsetneqq E_1 \subsetneqq E_2 \subsetneqq \cdots      \subsetneqq E_{k}  \subsetneqq E_{k+1}   \subsetneqq \cdots   \subsetneqq E_{p-2}   \subsetneqq E_{p-1} =\pi_{\ast}\mathcal{O}_X,
\end{gather}
as claimed.

By generic flatness there exists a nonempty open subset $V \subset Y$ such that $\pi \colon \pi^{-1}(V) \rightarrow V$ is flat and finite of degree $p$. Then the restriction $\pi_{\ast}\mathcal{O}_X|_V$ is locally free of rank $p$. Then from~\ref{filtration} it follows that $E_k$ has rank $k+1$. Moreover, from the exact sequence
\[
0 \rightarrow E_{k-1} \rightarrow E_k \rightarrow L_k \rightarrow 0\]
it follows that $L_k$ has rank 1.

Next I will show that $E_k$ is a reflexive sheaf on $Y$. Since $X$ is normal, $Y$ is also normal and hence by~\cite{Ha94} it suffices to show that $E_k$ is $S_2$. This is a local result and therefore we can assume that $X$ and $Y$ are both affine.  Let $W \subset Y$ be an open set such that $\mathrm{codim}(Y-W,Y)\geq 2$. Let $a \in E_k(W)$. Then $a \in \mathcal{O}_X(\pi^{-1}(W))$ and $D^{k+1}a=0$. Since $X$ is normal and $\mathrm{codim}(X-\pi^{-1}(W),X )\geq 2$, the restriction map $i_W \colon H^0(X,\mathcal{O}_X) \rightarrow \mathcal{O}_X(\pi^{-1}(W))$ is an isomorphism. Hence there exists $b \in H^0(X,\mathcal{O}_X)$ such that $i_W(b)=a$. Then since $i_W(D^{k+1}b)=D^{k+1}a=0$, it follows that $D^{k+1}b=0$ and hence $b\in E_k(Y)$. This implies that the natural map $H^0(Y,E_k) \rightarrow H^0(W, E_k)$ is an isomorphism and hence $H^1_Z(Y,E_k)=0$, where $Z=Y-W$. Therefore $E_k$ is $S_2$ as claimed.

Let $\phi_k \colon E_k \rightarrow \mathcal{O}_Y$ be the map defined locally by $\phi_k(a)=D^{k-1}a$. Then $\mathrm{Ker}(\phi_k)=E_{k-1}$ and therefore there is an injection of $\mathcal{O}_Y$-modules \[
L_k=E_k/E_{k-1} \rightarrow \mathcal{O}_Y.
\]
Hence $L_k$ is torsion free for all $k$. Moreover this shows that
\[
H^0(Y,L_1^{[-1]})=\mathrm{Hom}_Y(L_1,\mathcal{O}_Y) \not=0,
\]
as claimed in the second part of the theorem.

Let $\tau_k \colon E_1^{\otimes k} \rightarrow \pi_{\ast}\mathcal{O}_X$ be the map defined locally by $\tau_k(a_1 \otimes \cdots \otimes a_k)=a_1\cdots a_k$. Then an easy calculation shows that $D^k(a_1\cdots a_k)=0$ and hence $\tau_k$ induces a map $E_1^{\otimes k}\rightarrow E_k$ which induces a map $L^{\otimes k}_1 \rightarrow L_k$. Taking double duals we get a map \[
\tau_k \colon L_1^{[k]} \rightarrow L_k^{[1]}.
\]
Next I will show that the support of the kernel and cokernel of $\tau_k$ are in the fixed locus of the $\alpha_p$-action. Equivalently that $\tau_k$ is an isomorphism at any point not fixed by the $\alpha_p$-action. This a local result at points not fixed by $D$. Hence, by considering Proposition~\ref{subsec2-prop3}, we may assume that $X =\mathrm{Spec} A$, where $(A,m_A)$ is a local ring. Then $Y = \mathrm{Spec} B$, where $B=A^D$. Moreover, since the $\alpha_P$-action induced by $D$ is free, $D(A) \not\subset m_A$.

\textbf{Claim:} There exists $z \in A$ such that $Dz=1$.

Since $D(A)\not\subset m_A$, there exists $a\in A$ such that $Da \not\in m_A$ and hence $Da \in A^{\ast}$ is a unit. I will show that there exist $b_i \in B$, $i=1,\ldots,p-1$, such that 
\begin{gather}\label{Dz}
z=b_1a+b_2a^2+\cdots +b_{p-1}a^{p-1}
\end{gather}
has the property that $Dz=1$. By applying $D$ successively $p-1$ times in~\ref{Dz} we get that the existence of such $z$ is equivalent to the existence of a solution of the system of equations
\begin{gather}\label{existence-of-z}
\sum_{k=\nu}^{p-1}k(k-1)\cdots(k-\nu+1)b_ka^{k-\nu}=c_{\nu},
\end{gather}
$\nu=1,2,\ldots,p-1$, and the elements $c_{\nu}\in A$ are defined inductively from the relations $c_{\nu+1}=\frac{1}{Da}D(c_{\nu})$, $c_1=\frac{1}{Da}$. In fact,
\[
c_{\nu}=\left(\frac{1}{Da}D\right)^{\nu -1}\left(\frac{1}{Da}\right).
\]
Setting $\nu=p-1$ and by considering that by Wilson's theorem $(p-1)!=-1 \mod p$, we get that 
\begin{gather}\label{b}
b_{p-1}=-c_{p-1}=-\left(\frac{1}{Da}D\right)^{p-2}\left(\frac{1}{Da}\right).
\end{gather}
Now it is possible to inductively solve~\ref{existence-of-z} and find a solution $b_1,\ldots,b_{p-1}$. It remains to show that $b_i \in B$, for all $i$. I will show that $b_{p-1}\in B$. The rest follow by inverse induction.

Now by the Hochschild formula~\cite[Theorem 25.5]{Ma86}, for any $w \in A$,
\[
\left(wD\right)^p=w^pD^p+(wD)^{p-1}(w)\cdot D.
\]
Now from~\ref{b} and evaluating the Hochschield at $a$ for $w=1/Da$ and taking into consideration that $D^p=0$, we get that
\begin{gather*}
D(b_{p-1})=-D\left(\frac{1}{Da}D\right)^{p-2}\left(\frac{1}{Da}\right)=-Da \cdot \frac{1}{Da}D\left(\frac{1}{Da}D\right)^{p-2}\left(\frac{1}{Da}\right)=\\
=-Da\left(\frac{1}{Da}D\right)^{p-1}\left(\frac{1}{Da}\right)=-\left(\frac{1}{Da}D\right)^p(a)=\\
\left(\frac{1}{Da}D\right)^{p-2}\left(\frac{1}{Da}D\left(\frac{1}{Da}Da\right)\right)=0,
\end{gather*}
and therefore $b_{p-1}\in B$. Then by induction and using~\ref{existence-of-z} we get that $b_i \in B$ for all $i$. This completes the proof of the claim and hence there exists $z \in A$ such that $Dz=1$. Then by~\cite[Theorem 27.3]{Ma86}, the elements $z^k$, $k=0,1,\ldots, p-1$ form a basis for $A$ as a free $B$-module. Then it is easy to see that
\[
E_k=\{b_0+b_1z+\cdots b_k z^k|\; b_i \in B\}.
\]
Hence $E_k$ is free of rank $k+1$ away from the fixed locus of the $\alpha_p$ action, obtaining another proof that $E_k$ has rank $k+1$. Moreover, $L_k$ is a free $B$-module of rank 1 generated by $z^k$. But now it is clear that the map $\tau_k \colon L_1^k \rightarrow L_k$ is an isomorphism for all $k$.

Finally it is easy to verify that the map $\lambda \colon E_1 \rightarrow \mathcal{O}_Y$ given by $\lambda(a)=a^p$ is an $F$-splitting of the exact sequence
\[
0 \rightarrow \mathcal{O}_Y \rightarrow E_1 \rightarrow L_1 \rightarrow 0,
\]
where $F\colon X \rightarrow X$ is the absolute Frobenious map.

\end{proof}

\begin{remark}
The maps $\tau_k \colon L_1^{[k]} \rightarrow L_k^{[1]}$ are not necessarily surjective at the fixed points of the $\alpha_p$-action. For example, let $B=k[x,y]$ and $A=B[t]/(t^3-x)$, where $k$ is a field of characteristic 3. Let $D=t^2\frac{d}{dt}$. Then $D^3=0$ and $B=A^D$. An easy calculation shows that $E_1(D)=\{b_0+b_2t^2|\; b_0,b_1 \in B\}$ and $E_2=A$. Then $L_1=E_1(D)/B$ is a free $B$-module generated by $t^2$ and $L_2=E_2/E_1=A/E_1$ is a free $B$-module generated by $t$. Then the image of the map
\[
\tau_2 \colon L_1\otimes L_1 \rightarrow L_2\]
is the $B$-submodule of $L_2$ generated by $t^4=xt$. Hence $\tau_2$ is not surjective in this case.

\end{remark}

\begin{remark}
Proposition~\ref{section5-prop1} gives an elementary self contained proof of the structure theorem of $\alpha_L$-torsors given in Theorem~\ref{sec3-th1}.
\end{remark}

The next propositions show how to reconstruct $X$ from the data given in Proposition~\ref{section5-prop1}.

\begin{proposition}\label{section5-prop2}
Let $Y$ be an integral normal scheme of finite type over a perfect field $k$ of characteristic $p>0$. Let 
\begin{gather}\label{ex-sequence-sec5-prop2}
0 \rightarrow \mathcal{O}_Y \stackrel{i}{\rightarrow} E \rightarrow L \rightarrow 0
\end{gather}
be a short exact sequence of $\mathcal{O}_Y$-modules such that 
\begin{enumerate}
\item $E$ is reflexive of rank 2 and $L$ is torsion free of rank 1. 
\item There exists an $F$-splitting $\lambda \colon E \rightarrow \mathcal{O}_Y$ of the exact sequence~\ref{ex-sequence-sec5-prop2}.
\item There exists a nonzero section $s \in H^0(Y,L^{[-1]})$. 
\end{enumerate}
Let $\phi \colon S(E) \rightarrow \mathcal{O}_Y$ be the map defined locally by $\phi(a)=a^p$, if $a\in S^0(E)=\mathcal{O}_Y$, and $\phi(x_1\cdots x_n)=\lambda(x_1)\cdots \lambda(x_n)$, if $x_1\cdots x_n \in S^n(E)$, $n \geq 1$. Let \[
R(E,\lambda)=(S(E)/I)^{[1]}
\]
where $I \subset S(E)$ be the ideal sheaf generated locally by $a-i(a), a\in \mathcal{O}_Y$, and $\mathrm{Ker}(\phi)$. Then the map
\[
\pi \colon X=\mathrm{Spec}\left( R(E,\lambda)\right) \rightarrow Y
\]
is a finite purely inseparable map of degree $p$ and is an $\alpha_L$-torsor over a codimension 2 open subset of $Y$. Moreover, $X$ admits an $\alpha_p$-action and $Y$ is the quotient of $X$ by this action.
\end{proposition}

\begin{proof}
Let $Sing(Y)$ be the singular part of $Y$. Since $k$ is perfect, $Sing(Y)$ is a closed subset of $Y$ of codimension $\geq 2$. Let $Z \subset Y$ be the subset of $Y$ where $L$ is not free. Since $L$ is torsion free and $Y$ normal, $Z$ is closed of codimension $\geq 2$. Let $U=Y-(Sing(Y)\cup Z)$. Then $U$ is open and of codimension $\geq 2$. Moreover the restriction to $U$ of $E$ and $L$ are locally free. Now since $X$ is by construction a scheme satisfying the $S_2$ condition, all the statements of the proposition can be checked over $U$. In particular, if there is an $\alpha_p$-action on $\pi^{-1}(U)$ whose quotient is $U$, this action is induced by a vector field $D_u$ on $\pi^{-1}(U)$ such that $D^p_u=0$. Then since $\pi^{-1}(U)$ is of codimension $\geq 2$ in $X$, $D_u$ extends to a vector field $D$ on $X$ such that $D^p=0$. This vector field induces an $\alpha_p$-action on $X$ that extends the one on $\pi^{-1}(U)$. 

Hence in order to prove the proposition we may assume that $Y$ is smooth and $E$ and $L$ locally free. But then the construction of $\pi \colon X \rightarrow Y$ explained in the proposition is the $\alpha_{L^{-1}}$-torsor over $Y$ defined by the exact sequence~\ref{ex-sequence-sec5-prop2}. The explicit structure of such a torsor was described in Proposition~\ref{sec3-th1} and the proposition follows from this.
\end{proof}

\begin{theorem}\label{sec5-th1}
Let $X$ be an normal integral scheme of finite type over a perfect field $k$ of characteristic $p>0$. Suppose that $X$ admits a nontrivial $\alpha_p$-action. Let $\pi \colon X \rightarrow Y$ be the quotient. Then there exists a factorization
\[
\xymatrix{
X \ar[r]^{\phi}\ar[d]_{\pi} & Z \ar[dl]^{\delta} \\
Y   & \\
}
\]
such that 
\begin{enumerate}
\item \[
\delta \colon Z =\mathrm{Spec}_Y \left(R(E,\lambda)\right) \rightarrow Y
\]
is the purely inseparable map of degree $p$ defined by an $F$-split short exact sequence of $\mathcal{O}_Y$-modules
\[
0 \rightarrow \mathcal{O}_Y \stackrel{i}{\rightarrow} E \rightarrow L \rightarrow 0
\]
where  $E$ is reflexive of rank 2, $L$ torsion free of rank 1, $\lambda \colon E \rightarrow \mathcal{O}_Y$ an $F$-splitting of the sequence and $H^0(Y,L^{[-1]})\not= 0$.
\item $\phi \colon X \rightarrow Z$ is the normalization of $Z$. 
\item If $p=2$ or $\alpha_p$ acts freely on $X$, then $\phi$ is an isomorphism.
\end{enumerate}

\end{theorem}
\begin{proof}
By Proposition~\ref{sec1-prop1}, the $\alpha_p$ action on $X$ is induced by a nontrivial vector field $D$ on $X$ such that $D^p=0$. Then by Proposition~\ref{section5-prop1}, there exists an $F$-split exact sequence
\[
0 \rightarrow \mathcal{O}_Y \stackrel{i}{\rightarrow} E \rightarrow L \rightarrow 0,\]
where $E=E_1(D)\subset \pi_{\ast}\mathcal{O}_X$ is a rank 2 reflexive sheaf on $Y$ and $L=E/\mathcal{O}_Y$ is a rank 1 torsion free sheaf on $Y$. Moreover, $H^0(Y,L^{[-1]}) \not=0$. Let $\lambda \colon E \rightarrow \mathcal{O}_Y$ be an $F$-splitting of the sequence. Then let $R(E,\lambda)$ be the sheaf of $\mathcal{O}_Y$-algebras as constructed in Proposition~\ref{section5-prop2} and \[
\delta \colon Z=\mathrm{Spec}_Y\left( R(E,\lambda)\right) \rightarrow Y
\]
the purely inseparable map of degree $p$ defined by $R(E,\lambda)$. Let \[
\Psi \colon S(E) \rightarrow \pi_{\ast}\mathcal{O}_X
\]
be the map of $\mathcal{O}_Y$-algebras defined by the maps
\[
\Psi_m \colon S^m(E)\rightarrow \pi_{\ast}\mathcal{O}_X\]
given locally by $\Psi_m(x_1\cdots x_m)=x_1 \cdots x_m$, for any $m \geq 0$. Then clearly $\Psi(a-i(a))=0$, for any $a \in \mathcal{O}_Y$ and therefore the ideal $(a-i(a)|\; a\in \mathcal{O}_Y) \subset \mathrm{Ker}(\Psi)$. Let $\Phi \colon S(E) \rightarrow \mathcal{O}_Y$ be the map defined locally by $\Phi(a)=a^p$, for $a \in S^0(E)=\mathcal{O}_Y$ and on $S^n(E)$ by $\Phi(x_1x_2\cdots x_n)=\lambda(x_1)\lambda(x_2)\cdots \lambda(x_n)$, $m\geq 1$. Suppose that $w=\sum_ia_ix_{1,i}\cdots x_{n_i,i} \in S(E)$ such that $\Phi(w)=0$. Then by the definition of $\lambda$ and $\Phi$ this implies that \[
0=\sum_ia_i^p\lambda(x_{1,i})\cdots \lambda(x_{n_i,i})=\left(\sum_ia_ix_{1,i}\cdots x_{n_i,i}\right)^p.\]
Therefore $\sum_ia_ix_{1,i}\cdots x_{n_i,i}=0$ and hence $w\in \mathrm{Ker}(\Psi)$. Hence there is a map of $\mathcal{O}_Y$-algebras
\begin{gather}\label{psi-map}
\Psi \colon R(E,\lambda) \rightarrow \pi_{\ast}\mathcal{O}_X. 
\end{gather}
Hence there exists a factorization \[
\xymatrix{
X \ar[r]^{\phi}\ar[d]_{\pi} & Z \ar[dl]^{\delta} \\
Y   & \\
}
\]
as claimed in the theorem. Since both $\pi$ and $\delta$ are finite of degree $p$, it follows that $\phi$ has degree zero. Hence since $X$ is normal, $\phi\colon X \rightarrow Z$ is the normalization of $Z$. 

Suppose that $p=2$. Then $E=\pi_{\ast}\mathcal{O}_X$ and hence $\phi$ is an isomorphism. 

Suppose that $\alpha_p$ acts freely on $X$. I will show that in this case the map $\Psi$ in ~\ref{psi-map} is an isomorphism. This is a local result. Hence we can assume that $X =\mathrm{Spec} A$ and $Y = \mathrm{Spec} B$, where $B=A^D$. Then $A$ is a free $B$-module of rank $p$ and moreover, from the proof of Proposition~\ref{section5-prop1} it follows that there exists a $z\in A$  such that $Dz=1$ and $1,z,\ldots,z^{p-1}$ generate $A$ as a free $B$-module. Then \[
E=E_1=\{a\in A|\; D^2a=0\}=\{b_0+b_1z|\; b_0,b_1\in B\}.
\]
Then $\lambda \colon E \rightarrow B$ is defined by $\lambda(b_0+b_1z)=b_0^p+b_1^pz^p$.  It is now easy to see that \[
R(E,\lambda)\cong \frac{B[u]}{(u^p-z^p)}\]
and that the map $\Psi \colon R(E,\lambda) \rightarrow A$ is defined by $\Psi(f(u))=f(z)$. But since $A$ is generated by $1,z,\ldots,z^{p-1}$, this is an isomorphism. Hence $\Psi$ is an isomorphism and therefore $\phi$ is as well, as claimed in the theorem.
\end{proof}

\section{Adjunction formulas.}
Let $X$ be a scheme defined over an algebraically closed field $k$ of characteristic $p>0$. Suppose that $X$ admits either a $\mu_p$ or an $\alpha_p$ action and let $\pi \colon X \rightarrow Y$ be the quotient. The purpose of this section is to obtain adjunction formulas for $\pi$. If $X$ is normal then adjunction formulas are provided by Proposition~\ref{general-structure-theory} and Theorems~\ref{sec4-th1},~\ref{sec5-th1}. The purpose of this section is to give adjunction formulas in the case when $X$ is not necessarily normal. One possible approach to the non-normal case would be to take the normalization $\nu \colon \tilde{X} \rightarrow X$ of $X$, lift the $\alpha_p$ or $\mu_p$ action on $\tilde{X}$ and then take the quotient again. This is equivalent to lifting the derivation inducing the action on the normalization of $X$. However this is not always possible for the reason that unlike the characteristic zero case where derivations lift to the normalization~\cite{Sei66}, this is not true anymore in positive characteristic.

The next theorem gives an adjunction formula in the case when $X$ has at worst normal crossing singularities in codimension one. This case is sufficient as far as the problem of compactifying the moduli space of canonically polarized surfaces is concerned, the main motivation of this work.
   
\begin{theorem}\label{adjunction}
Let $X$ be an integral scheme defined over an algebraically closed field $k$ of characteristic $p>0$. Suppose that $X$ has a dualizing sheaf $\omega_X$, satisfies Serre's condition $S_2$ and that it has at worst normal crossing singularities in codimension one. Suppose that $X$ admits either a $\mu_p$ or an $\alpha_p$ action and let $\pi \colon X \rightarrow Y$ be the quotient of $X$ by the action. Then $Y$ has a dualizing sheaf $\omega_Y$ and \[
\omega_X=\left(\pi^{\ast}\omega_Y \otimes I_{\mathrm{fix}}^{[1-p]}\right)^{[1]},
\]
where $I_{\mathrm{fix}}$ is the ideal sheaf of the fixed locus of the action.
\end{theorem}   

\begin{proof}
By Proposition~\ref{sec1-prop1}, the action is induced by a global vector field $D$ such that either $D^p=D$ (the $\mu_p$-action case) or $D^p=0$ (the $\alpha_p$-action case).  I will only do the case when $D^p=0$, which corresponds to the $\alpha_p$ action. The other is similar and is ommited.

The map $\pi$ is purely inseparable of degree $p$. Therefore there is a factorization 
\[
\xymatrix{
X\ar[dr]_{\pi} \ar[rr]^{F} & & X^{(p)} \\
                            & Y \ar[ur]_{\delta} &
}
\]
where $F \colon X \rightarrow X^{(p)}$ is the relative FrobeniouFrobenius. Since $X$ has a dualizing sheaf, $X^{(p)}$ also has a dualizing sheaf and hence since $\delta$ is finite, $Y$ has a dualizing sheaf $\omega_Y$ as well. 

The proof of the theorem will be in several steps.

\textbf{Step 1.} I will show that $Y$ has Gorenstein singularities in codimension 1. 

Let $U =X-Fix(D)$, where $Fix(D)$ is the fixed locus of $D$. Then $U$ is open and nonempty. Let $V=\pi(U)\subset Y$. Since $D$ has no fixed points in $U$, the restriction $\pi \colon U \rightarrow V$ is faithfully flat. Let $P \in U$ be a Gorenstein point and $Q=\pi(P)\in V$. I will show that $Q\in V$ is also Gorenstein. Let $A=\mathcal{O}_{X,P}$ and $B=\mathcal{O}_{Y,Q}$. Then $B$ is Gorenstein if and only if 
$\mathrm{injdim}(B)=\dim B$. Let $n=\dim B$. Then $n=\dim A$ and since $A$ is Gorenstein, $\mathrm{injdim}(A)=n$. Let $M$ be a $B$-module. Since $B \rightarrow A$ is faithfully flat, \[
\mathrm{Ext}^{n+1}_B(M,B)\otimes_BA=\mathrm{Ext}^{n+1}_A(M\otimes_B A,A)=0,
\]
since $\mathrm{injdim}(A)=n$. Since $A$ is faithfully flat over $B$ it follows that $\mathrm{Ext}^{n+1}_B(M,B)=0$ and therefore $B$ is Gorenstein.

Let now $P \in X$ be a codimension 1 Gorenstein point that is also fixed by $D$.  Let $Q=\pi(P)$. Then 
\[
\hat{\mathcal{O}}_{X,P} \cong \frac{L[[x,y]}{(xy)},
\]
where $L$ is a coefficient field of $\hat{\mathcal{O}}_{X,P}$. Moreover, by Proposition~\ref{subsec2-prop3}, $D$ extends to a $k$-derivation $\hat{D}$ of $\hat{\mathcal{O}}_{X,P}$. Let $m_Q$, $m_P$ be the maximal ideals of $\hat{\mathcal{O}}_{Y,Q}$ and $\hat{\mathcal{O}}_{X,P}$, respectively. 

\textit{Case 1.} Suppose that $m_Q \hat{\mathcal{O}}_{X,P}=m_P$. In this case it is not hard to see that after a change of ccordinates, $\hat{\mathcal{O}}_{X,P} \cong L[[x,y]/(xy)$, and $m_Q=(x,y)$, i.e., 
$x,y \in \hat{\mathcal{O}}_{Y,Q}$. Therefore, $\hat{D}=f(x,y)D_L$, where $D_L$ is a derivation of $L$. Then 
\[
\hat{\mathcal{O}}_{Y,Q}\cong \frac{M[[x,y]]}{(xy)},
\]
where $M=L^{D_L}=\{a\in L |\; D_L(a)=0\}$. Hence in this case, $\hat{\mathcal{O}}_{Y,Q}$ is a normal crossings singularity and therefore Gorenstein.

\textit{Case 2.} Suppose that $m_Q \hat{\mathcal{O}}_{X,P}\not= m_P$. Then since $\pi$ has degree $p$, the natural map $ \hat{\mathcal{O}}_{Y,Q} \rightarrow \hat{\mathcal{O}}_{X,P}$ is an isomorphism. Therefore $L \subset \hat{\mathcal{O}}_{Y,Q}$. Then after a change of coordinates
\[
\hat{D}=xf(x)\frac{\partial}{\partial x}+yg(y)\frac{\partial}{\partial y}.
\]
Then if $(f(x),g(y) )\not= (0,0)$, then 
\[
\hat{\mathcal{O}}_{Y,Q}=\frac{L[[x^p,y^p]]}{(x^py^p)},
\]
and if one of $f(x)$, $g(y)$ is zero, then \[
\hat{\mathcal{O}}_{Y,Q}\cong \frac{L[[x^p,y]]}{(x^py)}.
\]
In any case, $Q \in Y$ is also normal crossings and hence Gorenstein.

\textbf{Step 2.} Since $X$ satisfies Serre's condition $S_2$, by Proposition~\ref{sec2-prop3} $Y$ also satisfies it. Hence since both $X$ and $Y$ are Gorenstein in codimension 1 and the formula claimed in the theorem is determined over an open set of codimension $\geq 2$, it suffices to assume that both $X$ and $Y$ are Gorenstein.

From now on I assume that $X$ and $Y$ are Gorenstein. Then from the proof of Theorem~\ref{sec3-th1} it follows that
\[
\omega_X=\pi^{\ast}\omega_Y \otimes \pi^{!}\mathcal{O}_Y,
\]
where $\pi^{!}\mathcal{O}_Y=\mathcal{H}om_Y(\pi_{\ast}\mathcal{O}_X,\mathcal{O}_Y)$. Let $f_D \in H^0(X,\pi^{!}\mathcal{O}_Y)$ be the section defined locally by $f_D(a)=D^{p-1}a$. This defines an injection
\[
\sigma \colon \mathcal{H}om_X(\pi^{!}\mathcal{O}_Y,\mathcal{O}_X) \rightarrow \mathcal{O}_X.
\]
Hence $\sigma$ induces an isomorphism of $(\pi^{!}\mathcal{O}_Y)^{\ast}$ with a principal ideal sheaf of $\mathcal{O}_X$. In fact I will show that $\sigma$ induces an isomorphism of $(\pi^{!}\mathcal{O}_Y)^{\ast}$ with $I_{\mathrm{Fix}}^{[p-1]}$. Recall that by its definition, $I_{\mathrm{Fix}}=(D(\mathcal{O}_X))$. 

\textbf{Claim 1.} \[
(D(\mathcal{O}_X))^{p-1} \subset \mathrm{Im}(\sigma).
\]

This  can be checked locally. So suppose that $X$ and $Y$ are affine. Say, $X=\mathrm{Spec}A$ and $Y=\mathrm{Spec}B$, where $B=A^D=\{a\in A|\; Da=0\}$. Since $X$ and $Y$ are Gorenstein, $\pi^{!}\mathcal{O}_Y$ is invertible. Let $g \in \mathrm{Hom}_B(A,B)$ be a generator as a free $A$-module. Then there exists $\alpha_0 \in A$ such that $f_D=\alpha_0 \cdot g$ and therefore $\mathrm{Im}(\sigma)=(\alpha_0)$.

\textbf{Claim 1.1} For any $a_1,a_2, \ldots,a_p \in A$, 
\begin{gather}\label{claim-1-1}
g(a_0a_1\cdots a_p)=
-2\sum_{1\leq i \leq p}a_iDa_1\cdots \hat{D}a_i\cdots Da_p +\\
\sum_{k=1}^{p-1}\left(\sum_{1\leq i_1\leq \cdots \leq i_k \leq p}(-1)^k(2k-1)a_{i_1}\cdots a_{i_k}g(a_o\cdots \hat{a}_{i_1}\cdots 
\hat{a}_{i_k} \cdots a_p)\right).\nonumber
\end{gather} 

I proceed to prove the claim. According to Leibniz formula for derivations
\begin{gather}\label{Leibniz}
g(a_0a_1\cdots a_p)=D^{p-1}(a_1\cdots a_p)=\sum_{k_1+k_2+\cdots k_p=p-1}\binom{p-1}{k_1,\cdots,k_p}D^{k_1}a_1\cdots D^{k_p}a_p.
\end{gather}
From this it follows that in order to prove the claim it sufffices to show that the terms 
\begin{equation}\label{term-to-compare}
a_{i_1}\cdots a_{i_k}D^{\nu_{k+1}}a_{i_{k+1}}\cdots D^{\nu_p}a_{i_p}
\end{equation}
appear with the same coefficient ($\mod p$)  on both sides of the equation~\ref{claim-1-1}, for all $1\leq i_s\leq p$, $1\leq s \leq k$, $\nu_{k+1}+\cdots +\nu_p=p-1$. 

Suppose that $k=1$. Then the term $a_{i}Da_1\cdots \hat{D}a_i \cdots Da_p$ appears $p-1=-1 \mod p$ times on the left hand side of~\ref{claim-1-1} and $-2-(p-1)=-p-1=-1 \mod p$ on the right hand side of~\ref{claim-1-1}.

Suppose that $k \geq 2$. On the left hand side of~\ref{claim-1-1}, the term~\ref{term-to-compare} appears $\lambda=\binom{p-1}{\nu_{k+1},\ldots,\nu_p}$ times. On the right hand side it appears 
\[
\left(\sum_{s=1}^k(-1)^s(2s-1)\binom{k}{s}\right)\lambda.
\]
I will show that 
\begin{equation}\label{combinatorial-1}
\sum_{s=1}^k(-1)^s(2s-1)\binom{k}{s}=1
\end{equation}
and therefore~\ref{claim-1-1} holds. First notice that from the formula
\[
0=\left((1+(-1)\right)^k=\sum_{s=0}^{k}(-1)^s\binom{k}{s}
\]
it follows that
\begin{equation}\label{comb-2}
\sum_{s=1}^{k}(-1)^s\binom{k}{s}=-1.
\end{equation}
Now let
\[
f(x)=(1-x)^k=\sum_{s=0}^k(-1)^sx^s\binom{k}{s}.
\]
Then \[
f^{\prime}(x)=\sum_{s=1}^k(-1)^ssx^{s-1}\binom{k}{s}
\]
and therefore
\[
f^{\prime}(1)=\sum_{s=1}^k(-1)^ss\binom{k}{s}
\]
But since $f^{\prime}(1)=0$, it follows that
\begin{equation}\label{comb-3}
\sum_{s=1}^k(-1)^ss\binom{k}{s}=0.
\end{equation}
Taking into consideration~\ref{comb-2} and~\ref{comb-3}, it follows that
\[
\sum_{s=1}^k(-1)^s(2s-1)\binom{k}{s}=2\sum_{s=1}^k(-1)^ss\binom{k}{s}-\sum_{s=1}^{k}(-1)^s\binom{k}{s}=0+1=1.
\]
This proves~\ref{combinatorial-1} and hence the Claim 1.1.

\textbf{Claim 1.2.} For any $a \in A$,
\begin{gather}\label{claim-1-2}
\left(Da\right)^{p-1}=-\left(a^{p-1}g(1)+\sum_{k=1}^{p-1}a^{p-1-k}g(a^k)\right)a_0.
\end{gather}
This relation shows that $I_{\mathrm{Fix}}^{[p-1]}\subset \mathrm{Im}(\sigma)\cong \mathcal{H}om_X(\pi^{!}\mathcal{O}_Y,\mathcal{O}_X)$.

In the formula~\ref{claim-1-1} put $a_1=a_0^{p-1}$ and $a_s=a^{s-1}$, for $2 \leq s \leq p$. Then 
\begin{gather}\label{form-after-1-1}
g(a_0a_1\cdots a_p)=a_0^pa^{\frac{p(p-1)}{2}}g(1)
\end{gather}
and
\[
a_{i_1}a_{i_2}\cdots a_{i_s}g(a_o\cdots \hat{a}_{i_1}\cdots \hat{a}_{i_s} \cdots a_p)=
\begin{cases}
a_0^pa^{\frac{p(p-1)}{2}-k}g(a^k), & \text{if} \; i_1 \not= 1 \\
a_0^{p-1}a^{\frac{p(p-1)}{2}-k}g(a_0a^k), & \text{if} \; i_1 = 1 
\end{cases}
\]
where $0 \leq k \leq p-1$ is such that $\sum_{\nu = 1}^s(i_{\nu}-1)=p\delta-k$, $\delta \geq 0$. The claim will be proved by counting the times ($\mod p$) that the terms $a_0^pa^{\frac{p(p-1)}{2}-k}g(a^k)$ and $a_0^{p-1}a^{\frac{p(p-1)}{2}-k}g(a_0a^k)$ appear in~\ref{claim-1-1}.

Let $d_k, \tilde{d}_k$ be the number of times that $a_0^{p-1}a^{\frac{p(p-1)}{2}-k}g(a_0a^k)$ and $a_0^pa^{\frac{p(p-1)}{2}-k}g(a^k)$ appear in the right hand side of~\ref{claim-1-1}, respectively. Then $d_k$ is the number of terms of the form $a_{i_1}a_{i_2}\cdots a_{i_{\nu}}g(a_o\cdots \hat{a}_{i_1}\cdots \hat{a}_{i_{\nu}} \cdots a_p)$ such that $2 \leq i_1 \leq i_2 \leq \cdots \leq i_{\nu} \leq p$ and $i_1+i_2+\cdots +i_{\nu}=-k \mod p$. Then it is not difficult to see that
\begin{gather}\label{d-k}
d_k=\sum_{s=2}^{p}(-1)^s(2s-1)d_k(s) \\
\tilde{d}_k=\sum_{s=1}^{p-1}(-1)^{s+1}(2s-1)d_k(s) \\\nonumber
\end{gather}
where 
\begin{gather*}
d_k(\nu)=\#\{(i_1,i_2,\ldots,i_{\nu}) |\; i_1+i_2+\cdots i_{\nu}=\nu-k\mod p, \;  2 \leq i_1 <i_2< \cdots <i_{\nu} \leq p \}=\\
\#\{(i_1,i_2,\ldots,i_{\nu}) |\; i_1+i_2+\cdots i_{\nu}=-k\mod p, \;  1 \leq i_1 <i_2< \cdots <i_{\nu} \leq p-1 \}.
\end{gather*}

\textbf{Claim 1.2.1.} 
\begin{enumerate}
\item $d_k=0 \mod p$,
\item $\tilde{d}_k=-2 \mod p$, $k\geq 1$,
\item $\tilde{d}_0=2 \mod p$.
\end{enumerate}
I will prove the first part of the claim. The other is proved similarly and is left to the reader. 

The first step is to get an explicit formula for $d_k(\nu)$. Let \[
\delta_k(\nu)=\#\{(x_1,x_2,\ldots,x_{\nu})|x_i \not= x_j, \text{for}\; i\not= j, x_1+x_2+\cdots x_{\nu}=k\mod p\}.
\]
It is not difficult to see that there is the following inductive formula
\[
\delta_k(\nu)=\frac{\nu!}{p}\binom{p}{\nu}-\nu\delta_k(\nu-1).
\]
From this it immediately follows that 
\begin{gather}\label{delta}
\delta_k(\nu)=\frac{\nu!}{p}\sum_{s=0}^{\nu-1}(-1)^s\binom{p}{\nu-s}.
\end{gather}
Let \[
\Omega_{\nu}=\{(x_1,x_2,\ldots,x_{\nu})|x_i \not= x_j, \text{for}\; i\not= j, x_1+x_2+\cdots x_{\nu}=k\mod p\}.
\]
Then $S_{\nu}$ acts on $\Omega_{\nu}$ by $\sigma \cdot (x_1,x_2,\ldots,x_{\nu})=(x_{\sigma(1)},x_{\sigma(2)},\ldots,x_{\sigma(\nu)})$, $\sigma \in S_{\nu}$. It is easy to see that the action is free and that all the orbits have the same number of elements. Moreover, in the orbit of any element $(x_1,x_2,\ldots,x_{\nu})$, there exists a unique $\sigma \in S_{\nu}$ such that $x_{\sigma(1)}<x_{\sigma(2)}< \cdots < X_{\sigma(\nu)}$. Then
\begin{gather}\label{d-k-1}
d_k(\nu)=\frac{|\Omega_{\nu}|}{|S_{\nu}|}=\frac{1}{\nu!}\delta_k(\nu)= \frac{1}{p}\sum_{s=0}^{\nu-1}(-1)^s\binom{p}{\nu-s}.
\end{gather}

From the equations~\ref{d-k},~\ref{d-k-1} it follows that $d_k(\nu)$, and hence $d_k$ is a sum of terms of the form $\binom{p}{s}$. Hence there $\lambda_i \in \mathbb{Z}$ such that \[
d_k=\sum_{s=1}^{p-1}\lambda_i\binom{p}{s}.
\]
I will show that $\lambda_s=(-1)^{s+1}\frac{1}{p}[p^2-s^2]$.

From the equation~\ref{d-k} and~\ref{d-k-1} it follows that
\begin{gather*}
\lambda_k=\frac{(-1)^k}{p}\left((2k+1)+(2k+3)+\cdots +(2k+2p-2k-1)\right)=\\
=\frac{(-1)^k}{p}\left(2k(p-k)+(p-k)^2\right)=\frac{(-1)^k)}{p}[p^2-k^2].
\end{gather*}
Hence
\begin{gather}\label{d-k-final}
d_k=\frac{1}{p}\sum_{s=1}^{p-1}(-1)^s(p^2-s^2)\binom{p}{s}=\\
p\sum_{s=1}^{p-1}(-1)^s\binom{p}{s} +\frac{1}{p}\sum_{s=1}^{p-1}(-1)^{s+1}s^2\binom{p}{s}.\nonumber
\end{gather}
Now considering that $\binom{p}{s}=\binom{p}{p-s}$, it follows that the coefficient of $\binom{p}{s}$ in the second summand is 
\[
(-1)^{s+1}s^2+(-1)^{ps+1}(p-s)^2=(-1)^{s+1}(2ps-p^2).
\]
Hence from~\ref{d-k-final} it follows that $d_k=0 \mod p$, as claimed. Therefore there are no terms of the form  $a_0^{p-1}a^{\frac{p(p-1)}{2}-k}g(a_0a^k)$ in the right hand side of~\ref{claim-1-1},

Next I consider the terms of the form $a_sDa_1\cdots \hat{D}a_s \cdots Da_p$. A straightforward calculation shows that
\begin{gather*}
a_sDa_1\cdots \hat{D}a_s \cdots Da_p =\frac{1}{s-1}a_0^{p-2}aa^{\frac{(p-1)(p-2)}{2}}Da_0(Da)^{p-2},
\end{gather*}
for $s \geq 2$. Hence
\begin{gather*}
\sum_{i=2}^pa_iDa_1\cdots \hat{D}a_i\cdots Da_p=
\left(1+\frac{1}{2}+\frac{1}{3}+\cdots+\frac{1}{p-1}\right)a_0^{p-2}aa^{\frac{(p-1)(p-2)}{2}}Da_0(Da)^{p-2}=\\
\left(1+2+\cdots (p-1)\right)aa^{\frac{(p-1)(p-2)}{2}}Da_0(Da)^{p-2}=\\
\frac{p(p-1)}{2}aa^{\frac{(p-1)(p-2)}{2}}Da_0(Da)^{p-2} =0 \; \text{mod p}
\end{gather*}
Therefore there are also no terms of the form 
\[
\sum_{i=2}^pa_iDa_1\cdots \hat{D}a_i\cdots Da_p
\]
in the right hand side of~\ref{claim-1-1}. Moreover 
\begin{gather}\label{a1da2}
a_1Da_2\cdots Da_p=-a_0^{p-1}a^{\frac{(p-1)(p-2)}{2}}(Da)^{p-1} 
\end{gather}
Now combining Claim 1.2.1,~\ref{form-after-1-1},~\ref{a1da2} and Claim 1.1 we get Claim 1.2.

Hence it has been proved that locally
\[
I_{\mathrm{fix}}^{p-1}=(D(A))^{p-1}=\left(a^{p-1}g(1)+\sum_{k=1}^{p-1}a^{p-1-k}g(a^k)\right)a_0\subset (a_0)=\subset (\pi^{!}\mathcal{O}_Y)^{\ast}.
\]
In order to prove equality of their double duals it suffices it check it at codimension one points of $X$. Hence we may assume that $A$ is the completion of the local ring of $X$ at a codimension one point. Then in order to prove equality I will explicitly give a generator $g$ of $\pi^{!}\mathcal{O}_Y=\mathrm{Hom}_B(A,B)$ and study the relation between $I_{fix}^{[p-1]}$ and $(\pi^{!}\mathcal{O}_Y)^{\ast}$. 

\textbf{Case 1.} $A$ is the completion of the local ring of $X$ at a smooth point of codimension one.

In this case, $A\cong L[[x]]$, where $L$ is a coefficient field of $A$. Then either $B=L[[x^p]]$ and $D=f(t)\frac{d}{dt}$, or $B=M[[x]]$ and $D$ is induced by an $M$-derivation of $L$, where $M \subset L$ is purely inseparable of degree $p$. In the first case $g=\frac{d^{p-1}}{dt^{p-1}}$ is a generator of  $\mathrm{Hom}_B(A,B)$ as a free $A$-module of rank one. In the second case, $D^{p-1}$ is a generator.

\textbf{Case 2.} $A$ is the completion of the local ring of $X$ at a normal crossing point of codimension one.

In this case $A=L[[x,y]]/(xy)$ and moreover,  as was explained in Step 1 of the proof, one of the following happens. 
\begin{enumerate}
\item  $D=xf(x)\frac{d}{dt}+yg(y)\frac{d}{dt}$  and $B=\frac{L[[x^p,y^p]]}{(x^py^p)}$, if $(f(x),g(y))\not=(0,0)$, or $B=\frac{L[[x^p,y]]}{(x^py)}$, if $f(y)=0$.
\item $D$ is induced by a derivation of $L$ and $B=L^D[[x,y]]/(xy)$.
\end{enumerate}
In the first case, $g=D_0^{p-1}-id$ is a generator of $\mathrm{Hom}_B(A,B)$ as a free $A$-module of rank one, where $D_0=x\frac{\partial}{\partial x}+y\frac{\partial}{\partial y}$, or $D=x\frac{\partial}{\partial x}$. In the second case, $D^{p-1}$ is a generator. The verification of these statements are rather tedious but straightforward and they are left to the reader.

Suppose we are in case 1. Then $A=L[[x]]$ and for $a=x$, $a^{p-1}g(1)+\sum_{k=1}^{p-1}a^{p-1-k}g(a^k)$ is a unit in $A$. Hence 
\[
(D(A))^{p-1}=\left(a^{p-1}g(1)+\sum_{k=1}^{p-1}a^{p-1-k}g(a^k)\right)a_0=(a_0)=\mathrm{Hom}_B(A,B)=(\pi^{!}\mathcal{O}_Y)^{\ast}.
\]

Suppose we are in case 2. Suppose that $D=xf(x)\frac{d}{dt}+yg(y)\frac{d}{dt}$. Then explicit calculations in~\ref{claim-1-2} show that 
 \begin{gather*}
I_{\mathrm{fix}}^{p-1}=(D(A))^{p-1}=(xy,x^{p-1}f(x)^{p-1},y^{p-1}g(y)^{p-1}) ,\\
(\pi^{!}\mathcal{O}_Y)^{\ast}=(xy,f(x)^{p-1}+g(y)^{p-1}).
\end{gather*}
Let $J=(xy,x^{p-1},y^{p-1})$. Then
\[
I_{\mathrm{fix}}^{p-1}=J \cdot  (\pi^{!}\mathcal{O}_Y)^{\ast}=J \otimes (\pi^{!}\mathcal{O}_Y)^{\ast}
\]
and therefore 
\[
J^{\ast\ast} \otimes (\pi^{!}\mathcal{O}_Y)^{\ast} =I_{\mathrm{fix}}^{[p-1]}.
\]
Now it is not hard to see that $J^{\ast\ast}=A$ and hence \[
 (\pi^{!}\mathcal{O}_Y)^{\ast} =I_{\mathrm{fix}}^{[p-1]}.
\]
This concludes the proof of the theorem.

\end{proof}                                                                                                                                                                                                                                                                                                                                                                                                                                                                                                                                                                                                                                                                                          

\end{document}